\documentclass[reqno]{amsart}
\usepackage{amsmath,amsthm,amssymb}
\usepackage{latexsym}
\usepackage{eucal}
\usepackage{fullpage}

\newtheorem{theorem}{Theorem}[section]
\newtheorem{lemma}{Lemma}[section]
\newtheorem{corollary}{Corollary}[section]

\theoremstyle{definition}

\newtheorem{open problem}{Open Problem}

\def\cal{\mathcal}
\let\Re=\undefined
\DeclareMathOperator{\Re}{Re}
\let\Im=\undefined
\DeclareMathOperator{\Im}{Im}

\begin{document}
\title[ \ldots The sharp estimates on the orthogonal polynomials
\ldots
 ]{The sharp estimates on the orthogonal polynomials from the Steklov class }
\author{A. Aptekarev, S. Denisov, D. Tulyakov }
\address{
\begin{flushleft}
University of Wisconsin--Madison\\  Mathematics Department\\
480 Lincoln Dr., Madison, WI, 53706, USA\\
  denissov@math.wisc.edu\\ \vspace{1cm}
Keldysh Institute for Applied Mathematics,
Russian Academy of Sciences\\
Miusskaya pl. 4, 125047 Moscow, RUSSIA\\aptekaa@keldysh.ru
\end{flushleft}
  }
 \maketitle

\maketitle

\maketitle
\begin{abstract}
Given any $\delta\in (0,\delta_0)$ with sufficiently small $\delta_0$, we define the Steklov class
$S_\delta$ to be the set of probability measures $\sigma$ on the
unit circle, such that $\sigma'(\theta)\geq\delta/(2\pi)>0$ at every
Lebesgue point of $\sigma$. One can define the orthonormal
polynomials $\phi_n(z)$ with respect to $\sigma \in S_\delta$. In
this paper we consider the following variational problem. Fix $n \in
\mathbb{N}$ and define $ M_{n,\delta}=\sup_{\sigma\in S_\delta}
\|\phi_n\|_{L^\infty(\mathbb{T})}$.  Our main result is the lower bound in  the following estimate
 \[
C(\delta)\sqrt n < M_{n,\delta}\leq \sqrt{\frac{n+1}{\delta}}
\]
\end{abstract} \vspace{1cm}

\large

\section{Introduction}

One version of the Steklov's problem  (see~\cite{1}, \cite{2}) is to
find bounds for the polynomial sequences
$\{P_{n}(x)\}_{n=0}^{\infty}$, which are orthonormal
\begin{equation}\label{1}
\int\limits_{-1}^{1}P_{n}(x)\,P_{m}(x)\,\rho(x)\,dx=\delta_{n,m}\;,\quad n,m=0,1,2\ldots
\end{equation}
with respect to the strictly positive weight $\rho$:
\begin{equation}\label{2}
\rho(x)\geqslant\delta>0\;,\quad x\in[-1,1]\;.
\end{equation}
In 1921, V.A.~Steklov made a conjecture that a sequence
$\{P_{n}(x)\}$ is bounded at any point $x\in(-1,1)$ provided that
the weight $\rho$ does not vanish on $[-1,1]$. This problem and some
related questions gave rise to an extensive research, see, e.g., the
survey paper \cite{2}. In 1979, Rakhmanov~\cite{3} disproved this
conjecture constructing a weight from the Steklov class \eqref{2},
for which
$$
\sup\limits_{n}\,\{|P_{n}(0)|\}=\infty\;.
$$
It is known (see, for
example~\cite{5}) that for any Steklov's weight the following bound
$$
|P_{n}(x)|=\overline{o}(\sqrt{n})
$$
holds true for any $x\in(-1,1)$. In his next paper ~\cite{4}, Rakhmanov proved that
for every $\epsilon>0$ and $x_0\in (-1,1)$ there is a weight  $\rho(x;x_{0},\varepsilon)$
from the Steklov class such that
the corresponding $\{P_{n}(x)\}$ are growing as
\begin{equation}\label{3}
|P_{n}(x_{0})|\geqslant n^{1/2-\varepsilon}\;,\quad n\in\Lambda\,,
\end{equation}
where $\Lambda$ is some sequence of natural numbers. In \cite{murman}, the size of the polynomials for the continuous weight was studied.

All Rakhmanov's counterexamples were obtained as corollaries of
the corresponding results for the polynomials $\{\phi_{n}\}$
orthonormal on the unit circle
\begin{equation}\label{4}
\int\limits_{0}^{2\pi}\phi_{n}\,\overline{\phi}_{m}\,d\sigma(\theta)=
\delta_{n,m}\;,\quad n,m=0,1,2\ldots\,,
\end{equation}
with respect to  measures from the Steklov class
$S_{\delta}$ defined as the class of probability measures $\sigma$ on the unit circle satisfying
\[
\sigma'\geq\delta/(2\pi)
\]
at every Lebesgue point.
An important role in this construction was played by the following
extremal problem. For a fixed $n$, define
\begin{equation}\label{6}
M_{n,\delta}=\sup\limits_{\sigma\in
S_{\delta}}\|\phi_{n}(z;\sigma)\|_{L^{\infty}(\mathbb{T})}
\end{equation}
The problem of estimating the size of $\phi_n$ is one the most basic and most well-studied in the Approximation theory. The Steklov's condition on the weight is very important: e.g., it is a natural way to normalize the weight and the corresponding polynomials on the arcs in $\mathbb{T}$. Indeed, if one multiplies  the weight by the constant $C$, the corresponding polynomials  will be multiplied by $C^{-1/2}$ and one can expect similar phenomenon on the arcs.  The Steklov's condition rules out the scalings like that. In spite of the importance of this problem, very few methods were available to handle questions of that type. In the current paper we suggest a method which, we believe, is general enough to study the variational problems for other situations when the constructive information on the weight is given.\bigskip

{\bf Remark.} Since $S_\delta$ is invariant under the rotation and
$
\{\phi_j(ze^{-i\theta_0},\mu)\}
$
are orthonormal with respect to $\mu(\theta-\theta_0)$,
we can always assume that $\|\phi_n\|_\infty$ is reached at point $z=1$. Therefore, we have
\[
M_{n,\delta}=\sup_{\mu\in S_\delta} |\phi_n(1,\mu)|
\]
\bigskip

One of the key results in~\cite{4} is the following inequality
\begin{equation}\label{7}
C\,\sqrt{\frac{n+1}{\delta \ln^{3}n}}\leqslant M_{n,\delta}\;,\quad
C>0\;.
\end{equation}
We also recall the well known estimate (see~\cite{5})
\begin{lemma} We have
\begin{equation}\label{8}
M_{n,\delta}\leqslant\sqrt{\frac{n+1}{\delta}}\;,\quad
n\in\mathbb{N}\;.
\end{equation}
\end{lemma}
\begin{proof}
Indeed,
\[
1=\int_{\mathbb{T}}|\phi_n|^2d\sigma\ge\delta/(2\pi)
\int_{\mathbb{T}}|\phi_n|^2d\theta
\]
so (\ref{8}) follows from
\[
\|\phi_n\|_{L^2(d\theta)}^2=\left\|\sum_{j=0}^n
c_jz^j\right\|_{L^2(d\theta)}^2=2\pi\sum_{j=0}^n |c_j|^2\leq
2\pi/\delta
\]
and Cauchy-Schwarz
\[
\|\phi_n\|_\infty\leq (n+1)^{1/2}\sqrt{\sum_{j=0}^n |c_j|^2}
\]
\end{proof}
{\bf Remark.} Notice that all we used in the proof is the normalization
$\|\phi_n\|_{2,\sigma}=1$ and the Steklov's condition on the measure.
The problem, though, is whether the orthogonality leads to further restrictions
on the size. \smallskip

 The purpose of the current paper is to get rid of the logarithmic
factor in the denominator in (\ref{7}) and thus prove the optimal
result.\bigskip

One can consider the monic orthogonal polynomials
$\Phi_n(z,\mu)=z^n+\ldots$ and the Schur parameters $\gamma_n$ so
that
\[
\phi_n(z,\mu)=\frac{\Phi_n(z,\mu)}{\|\Phi_n\|_\mu}
\]
and $\Phi_n(0,\mu)=-\overline{\gamma}_{n-1}$. If $
\rho_n=(1-|\gamma_n|^2)^{1/2} $ then
\[
\Phi_n(z,\mu)=\phi_n(z,\mu)\Bigl(\rho_0\cdot\ldots\cdot
\rho_{n-1}\Bigr)
\]
The Szeg\H{o} formula \cite{sim1} yields
\[
\exp\left(\frac{1}{4\pi}\int_{-\pi}^\pi\log (2\pi\mu'(\theta))
d\theta\right)=\prod_{j\geq 0}\rho_j
\]
So, for $\mu\in S_\delta$, we have
\[
\delta^{1/2}\leq \prod_{n\geq 0}  \rho_n  \leq 1
\]
and therefore
\[
\delta^{1/2}  |\phi_n(z,\mu)|\leq |\Phi_n(z,\mu)|\leq
|\phi_n(z,\mu)|, \quad z\in \mathbb{C}
\]
for any $\mu\in S_\delta$. Thus, we have
\[
\delta^{1/2}M_{n,\delta}\leq \sup_{\mu\in S_\delta} |\Phi_n(z,\mu)|
\leq M_{n,\delta}
\]
and for fixed $\delta$ the variational problems for orthonormal and
monic orthogonal polynomials are equivalent.\bigskip

The structure of the paper is as follows. The Section 2 contains
several auxiliary results: existence of the optimal measure and its
property along with asymptotics of $M_{n,\delta}$ in the small
$\delta$ regime. In Section 3, we obtain the main result of the
paper, i.e., we prove the lower bound $M_{n,\delta}>C(\delta) \sqrt
n$ for fixed $\delta$ and large $n$. The application to the
orthogonal polynomials entropy is presented in the Section 4. The
Appendixes contain some auxiliary results.
\smallskip

The proofs by Rakhmanov were based on the following formula for the
orthogonal polynomial that one gets after adding several point
masses to ``background" measure at the particular locations on the
circle (see  \cite{3}).

\begin{lemma}Let $\mu$ be a positive measure on $\mathbb{T}$,
$\Phi_n(z,\mu)$-- the corresponding monic orthogonal polynomials and
\[
K_n(\xi,z,\mu)=\sum_{l=0}^n \overline{\phi_j}(\xi,\mu)\phi_j(z,\mu)
\]
is the Christoffel-Darboux kernel, i.e.
\[
P(\xi)=\langle P(z),K_n(\xi,z,\mu)\rangle,\quad \deg P\leq n
\]

 Then, if $\xi_j\in\mathbb{T},
j=1,\ldots m, m\leq n$ are chosen such that
\begin{equation}\label{uslo-k}
K_{n-1}(\xi_j,\xi_l,\mu)=0, j\neq l
\end{equation}
 then
\begin{equation}
\Phi_n(z,\eta)=\Phi_n(z,\mu)-\sum_{k=1}^m
\frac{m_k\Phi_n(\xi_k,\mu)}{1+m_kK_{n-1}(\xi_k,\xi_k,\mu)}K_{n-1}(\xi_k,z,\mu)
\end{equation}
where
\[
\eta=\mu+\sum_{k=1}^m m_k\delta(\theta-\theta_k), \quad
z_k=e^{i\theta_k}, \quad m_k\geq 0
\]
\end{lemma}

The limitation that $\xi_j$ must be the roots of $K$ is quite
restrictive and the direct application of this formula with
background $d\mu=d\theta$ yields only logarithmic growth at best as
was showed in \cite{REU}.

We will use completely different approach. First, we will rewrite
the Steklov condition in the convenient form as some conditions that
involve Herglotz function and a polynomial. Then, we will present
this function and the polynomial and show that they satisfy the
necessary conditions. This allows us to have a good control on the
size of the polynomial itself and on the structure of the measure of
orthogonality.

Some notation used in the paper: the Cauchy kernel for the unit
circle is denoted by $C(z,\xi)$, i.e.
\[
C(z,\xi)= \frac{\xi+z}{\xi-z}, \quad \xi\in \mathbb{T}
\]
If the function is analytic in $\mathbb{D}$ and has a nonnegative real part there, then we will call it Caratheodory function. Given two positive functions $F_1$ and $F_2$ defined on $\cal{D}$, we write
$F_1\lesssim F_2$ if there is a constant $C$ (that might depend only on the fixed parameters) such that
\[
F_1<CF_2
\]
on $\cal{D}$. We write $F_1\sim F_2$ if
\[
F_1\lesssim F_2\lesssim F_1
\]
\bigskip

\section{Structure of the extremal measure and solution of the problem in the small $\delta$ regime}

In this section, we first address the problem of the existence of
maximizers, i.e., $\mu^*_n\in S_\delta$ for which
\begin{equation}\label{extre}
M_{n,\delta}=|\phi_n(1;\mu^*_n)|
\end{equation}
We will prove that these extremizers exist and will study their
properties.

\begin{theorem}
There are $\mu^*_n\in S_\delta$ for which (\ref{extre}) holds.
\end{theorem}
\begin{proof}
Suppose $\mu_k\in S_\delta$ is the sequence which yields the $\sup$,
i.e.
\[
|\phi_n(1;\mu_k)|\to M_{n,\delta}, \quad k\to\infty
\]
Since the unit ball is weak-($\ast$) compact, we can choose
$\mu_{k_j}\to \mu^*$ and this convergence is weak-($\ast$), i.e.
\[
\int fd\mu_{k_j}\to \int fd\mu^*, \quad j\to\infty
\]
for any $f\in C(\mathbb{T})$. In particular, $\mu^*$ is a
probability measure. Moreover, for any interval $(a,b)\subseteq
(-\pi,\pi]$, we have (assuming, e.g., that the endpoints $a$ and $b$
are not atoms for $\mu^*$):
\[
\int_{[a,b]} d\mu^*\geq \delta(b-a)/(2\pi)
\]
since each $\mu_{k_j}\in S_\delta$. This implies
${\mu^*}'>\delta/(2\pi)$ a.e. on $\mathbb{T}$. The moments of
$\mu_{k_j}$ will converge to the moments of $\mu^*$ and therefore
\[
\phi_n(1;\mu_{k_j})\to\phi_n(1;\mu^*)
\]
Therefore, $\mu^*\in S_\delta$ and $|\phi_n(1;\mu^*)|=M_{n,\delta}$.
\end{proof}

This argument gives existence of an extremizer. Although we do not
know whether it is unique, we can prove that every $d\mu^*$ must
have a very special form. We start with the following well-known
result attributed to Geronimus:

\begin{lemma}
Consider $\mu(t)=(1-t)\mu+t\delta(\theta)$ where $t\in (0,1)$. Then,
\begin{equation}\label{insert}
\Phi_n(z,\mu(t))=\Phi_n(z,\mu)-t\frac{\Phi_n(1,\mu)K_{n-1}(1,z,\mu)}{1-t+tK_{n-1}(1,1,\mu)}
\end{equation}
\end{lemma}
\begin{proof}
Notice that the r.h.s. is a monic polynomial of degree $n$. Then,
one checks that
\[
\langle {\rm r.h.s.},z^j\rangle_{\mu(t)}=0, \quad j=0,\ldots, n-1
\]
which yields orthogonality.
\end{proof}
The proof of the next statement can be found in (\cite{sim1},
theorem 10.13.3, formula (10.13.17))
\begin{lemma}
The following formula holds true
\[
\|\Phi_n(z,\mu(t))\|^2_{\mu(t)}=\|\Phi_n(z,\mu)\|^2_{\mu}
(1-t)\frac{1-t+tK_n(1,1)}{1-t+tK_{n-1}(1,1)}
\]
\end{lemma}
These two results give the following expression
\[
|\phi_n(1,\mu(t))|^2=|\phi_n(1,\mu)|^2
\frac{1-t}{(1-t+tK_{n-1}(1,1,\mu))(1-t+tK_n(1,1,\mu))}
\]

We will need the following lemma later on
\begin{lemma}
The following formulas are true
\begin{equation}\label{drv1}
\left.\frac{d}{dt}|\Phi_n(1,\mu(t))|^2\right|_{t=0}=-2K_{n-1}(1,1,\mu)|\Phi_n(1,\mu)|^2<0
\end{equation}
\begin{equation}\label{drv2}
\left.\frac{d}{dt}|\phi_n(1,\mu(t))|^2\right|_{t=0}=|\phi_n(1,\mu)|^2(1-K_n(1,1,\mu)-K_{n-1}(1,1,\mu))
\end{equation}
\end{lemma}
\begin{proof}
We have
\[
|\Phi_n(1,\mu(t))|^2=\frac{(1-t)^2|\Phi_n(1,\mu)|^2}{(1-t+tK_{n-1}(1,1,\mu))^2}
\]
Taking derivative at zero gives (\ref{drv1}). One gets (\ref{drv2})
similarly.
\end{proof}
\begin{corollary} If $n>n_0(\delta)\gg 1$ and $\mu^*_n$ is an extremizer,
then
\begin{equation}\label{minus}
\left.\frac{d}{dt}|\phi_n(1,\mu_n^*(t))|^2\right|_{t=0}<0
\end{equation}
\end{corollary}
\begin{proof}For fixed $\delta$, we have a bound
\[
M_{n,\delta}>C\sqrt{\frac{n+1}{\delta \log^3 n}}
\]
which implies that $|\phi_n(1,\mu_n^*(t))|\to\infty$ as
$n\to\infty$. Consequently, $K_n(1,1,\mu_n^*(t))\to\infty$ and the
formula (\ref{drv2}) finishes the proof.
\end{proof}

 Suppose we have a positive measure $\mu$ and the moments
\[
s_j=\int e^{ij\theta}d\mu=s_j^R+is_j^I, \quad j=0,1,\ldots
\]
Then the following formula is well-known
\begin{eqnarray*}
\Phi_n(z)=D_{n-1}^{-1} \left|
\begin{array}{cccc}
s_0 & s_1 & \ldots& s_n\\
s_{-1} & s_0 & \ldots& s_{n-1}\\
\ldots& \ldots& \ldots& \ldots\\
s_{-(n-1)} & s_{-(n-2)} &\ldots&s_1 \\
1 & z& \ldots& z^n
\end{array}
\right|\,\\ D_n=\det T_n, \quad
 T_n=\left[
\begin{array}{cccc}
s_0 & s_1 & \ldots& s_n\\
s_{-1} & s_0 & \ldots& s_{n-1}\\
\ldots& \ldots& \ldots& \ldots\\
s_{-(n-1)} & s_{-(n-2)} &\ldots&s_1 \\
s_{-n} & s_{-(n-1)}& \ldots& s_0
\end{array}
\right]
\end{eqnarray*}
Therefore, the functions $F_{1(2)}$ given by
\[
F_1(s_0,s_1^R,\ldots,s_n^I)=|\Phi_n(1)|^2,
\quad F_2(s_0,s_1^R,\ldots,s_n^I)=|\phi_n(1)|^2
\]
are the smooth functions of the variables $\{s_j^R,s_j^I\},
j=0,\ldots, n$ wherever they are defined. Consider $\Omega_n=\{s:
T_n(s)>0\}$. Clearly, if $s\in \Omega_n$, then there is a family of
measures $\mu$ which have $s$ as the first $n$ moments (the solution
to the truncated trigonometric moments problem).

We will need the following
\begin{lemma}
The function $F_1(s)$ does not have stationary points on $\Omega_n$.
If $n>n_0(\delta)\gg 1$ and $\{s_j^*\}$ are the moments of the
extremizer $\mu_n^*$ then this $s^*$ is not a critical point for
$F_2(s)$.
\end{lemma}
\begin{proof}
Indeed, suppose $\widehat s$ is a stationary point and $\widehat
\mu$ is one of the measures that generates it. Then, consider
$\mu(t)=(1-t)\widehat\mu+t\delta(\theta)$. The corresponding moments
$s_j(t)$ are linear functions in $t$. Therefore,
\[
g(t)=F_1(s_0(t),s_1^R(t),\ldots, s_n^I(t))
\]
is differentiable at $t=0$ and $g'(0)=0$ as $\widehat s$ is a
stationary point. This contradicts (\ref{drv1}). The argument for
$F_2$ is the same except that one gets the contradiction with
(\ref{minus}).
\end{proof}
{\bf Remark.} The proof actually shows that at least one of the
derivatives ${\partial F_{2}}/{\partial s_j^{R(I)}}$ is different
from zero and $j=1,2,\ldots n$. This is because the measure $\mu(t)$
has the same variation as $\mu$.
\begin{theorem}
If $\mu^*$ is a maximizer then it can be written in the following
form
\begin{equation}\label{form}
d\mu^*=\delta d\theta+\sum_{j=1}^N m_j\delta(\theta-\theta_j), \quad
1\leq N\leq n
\end{equation}
where $m_j\geq 0$ and $-\pi<\theta_1<\ldots<\theta_N\leq \pi$.
\end{theorem}
\begin{proof}
 Our variational problem is
extremal problem for a functional $F(s_{0},s_{1}^R,\ldots,s_{n}^I)$ on
the finite number of moments $\{s_{0},s_{1}^R,\ldots,s_{n}^I\}$ of the
measure from $S_{\delta}$. We can take
$$
F(s_{0},s_{1}^R,\ldots,s_{n}^I)=|\phi_{n}(1)|^2\;.
$$
The function $F$ is differentiable. Moreover,
\[
s_0=\int d\mu,\quad
s_j^R=\int \cos(j\theta) d\mu,\quad s_j^I=\int \sin(j\theta)d\mu
\]
Considering the moments as the functionals in $\mu$, we compute the derivative of $F$ at the point $\mu^*$ in the direction $\delta\mu$:
\[
dF=\int \left(\frac{\partial{F}}{\partial s_0}(s^*)+\frac{\partial F}{\partial s_1^R}(s^*)\cos(\theta)+\ldots+\frac{\partial{F}}{\partial s_n^I}(s^*)\sin(n\theta)\right) d(\delta \mu)
\]
Consider the trigonometric polynomial of degree at most $n$:
\[
T_n(\theta)=\frac{\partial{F}}{\partial s_0}(s^*)+\frac{\partial F}{\partial s_1^R}(s^*)\cos(\theta)+\ldots+\frac{\partial{F}}{\partial s_n^I}(s^*)\sin(n\theta)
\]
From the previous lemma and remark, we know that it is not identically constant. Let $M=\max T_n(\theta)$ and $\{\theta_j; j=1,\ldots, N\}$ are the points where $M$ is achived.
Clearly, $N\leq n$.

Now, if we find a smooth curve  $\mu(t), t\in (0,1]$ such that $\mu(t)\in S_\delta, \mu(1)=\mu^*$ and define
\[
H(t)=F(s_0(\mu(t)),s_1^R(\mu(t)),\ldots,s_n^I(\mu(t))
\]
then $H'(1)\geq 0$ as follows from the optimality of $\mu^*$.

Now, we will assume that the measure $\mu^*$ is not of the form (\ref{form}) and then will come to the contradiction by choosing the curve $\mu(t)$ in a suitable way.

We will first prove that the singular part of $\mu^*$ can be supported only at points $\{\theta_j\}$. Indeed, suppose we have
\[
\mu^*=\mu_1+\mu_2
\]
where $\mu_2$ is singular and supported away from $\{\theta_j\}$.
Consider smooth $p_1(t)$ and $p_2(t)$ defined on $(0,1]$ and
satisfying
\[
|\mu_1|+p_1(t)+p_2(t)|\mu_2|=1, \quad p_{1(2)}(t)\geq 0,\quad p_{1}(1)=0, \quad p_2(1)=1
\]
For example, one can take $p_1(t)=|\mu_2|(1-t),\, p_2(t)=t$.
Take $\mu(t)=\mu_1+p_1(t)\delta(\theta_1)+p_2(t)\mu_2$. We have $\mu(t)\in S_\delta$ and
\[
H'(1)=\int T_n(\theta)d\mu_2-|\mu_2|T_n(\theta_1)<0
\]
since $\theta_1$ is the point of global maximum for $T_n$ and $\mu_2$ is supported away from $\{\theta_j\}$ by assumption.
This contradicts with optimality of $\mu^*$ and so $\mu_2=0$.

We can prove similarly now that $(\mu^*)'=\delta$ a.e. Indeed, suppose
\[
\mu^*=\mu_1+\mu_2, \quad \mu_2=f(\theta)\chi_{\Omega}d\theta
\]
where $f(\theta)>\delta_1>\delta$ on $\Omega$, $|\Omega|>0$ and $\mu_1$ is supported on $\Omega^c$.
We consider the curve
\[
\mu(t)=\mu_1+p_1(t)\delta(\theta_1)+p_2(t)\mu_2(t)
\]
The choice of $p_{1(2)}$ is the same and then $\mu(t)\in S_\delta$ for $t\in (1-\epsilon, 1)$ provided that $\epsilon(\delta_1)$ is small.
The similar calculation gives $H'(1)<0$ and that gives a contradiction.

\end{proof}

The formula (\ref{insert}) expresses all monic polynomials resulted
from  adding one point mass to arbitrary measure at any location and
one can try to iterate it to get the optimal measure $d\mu^*$. That,
however, leads to very complicated analysis. Nevertheless, the
simple application of this formula shows that the total number of
mass points in the optimal $\mu^*$ must necessarily grow in
$n$.\bigskip

One can make a trivial observation that is $\mu$ is any positive
measure (not necessarily a probability one) and $\phi_n(z;\mu)$ is
the corresponding orthonormal polynomial, then
\begin{equation}\label{scaling}
\phi_n(z,\alpha\mu)=\alpha^{-1/2}\phi_n(z;\mu)
\end{equation}
for every $\alpha>0$. The monic orthogonal polynomials, though, stay
unchanged
\[
\Phi_n(z,\alpha\mu)=\Phi_n(z;\mu)
\]
\bigskip

Now, consider the modification of the problem: we define
\[
\widetilde M_{n,\delta}=\sup_{\mu'>\delta/(2\pi)}
\|\phi_n(z;\mu)\|_\infty=\sup_{\mu'>\delta/(2\pi)} |\phi_n(1;\mu)|
\]
i.e. we drop the requirement for the measure $\mu$ to be a
probability measure. In this case, the upper estimate for
$\widetilde M_{n,\delta}$ stays the same (same proof)
\[
\widetilde{M}_{n,\delta}\leq \sqrt{\frac{n+1}{\delta}}
\]
It turns out that the sharp lower bound in this case can be easily
obtained and so we get
\begin{theorem}
We have
\[
\widetilde M_{n,\delta}=\sqrt{\frac{n+1}{\delta}}
\]
\end{theorem}
\begin{proof}
Consider
\begin{equation}\label{19}
d\sigma=\delta/(2\pi)
d\theta+\sum_{k=1}^{n}m_{k}\,\delta(\theta-\theta_{k})\;,
\quad\theta_{k}=\frac{k}{n+1}2\pi\,,\quad  k=1,\ldots,n\,,
\end{equation}
We assume that all $m_k\geq  0$. Consider
\begin{equation}\label{21}
\Pi_{n}(z)=\prod\limits_{k=1}^{n}(z-\varepsilon_{k})\;,\quad\varepsilon_{k}=e^{i\theta_{k}}
\end{equation}
and one gets: $ \Pi_n(z)=1+z+\ldots+z^n,\quad
\|\Pi_n\|^2_\sigma=\delta(n+1). $ We define now
\[
\Phi_n=\Pi_n+Q_{n-1}
\]
where $Q_{n-1}(z)=q_{n-1}z^{n-1}+\ldots+q_1z+q_0$ is chosen to
guarantee the orthogonality $\langle \Phi_n,z^j\rangle_{\sigma}=0,
j=0,\ldots, n-1$. Suppose now that $m_k=m$ for all $k$. Then, we
have the following equations
\[
\delta+\delta q_j+m\sum_{l=0}^{n-1}q_l\sum_{k=1}^n
\epsilon_k^{l-j}=0, \quad j=0,\ldots, n-1
\]
Then, since
\[
\sum_{k=0}^n \epsilon_k^d=0, \quad d\in \{-n,\ldots,-1,1,\ldots,n\}
\]
we get
\[
\delta+\delta q_j+m(n+1)q_j -m\sum_{l=0}^{n-1} q_l=0, \quad
j=0,\ldots, n-1
\]
and
\[
q_j=-\frac{\delta}{\delta+m}, \quad j=0,\ldots n-1
\]
Thus,
\[
\Phi_n(z)=1+\ldots+z^n-\frac{\delta}{\delta+m}(1+\ldots+z^{n-1})=
\frac{m}{\delta+m}\Pi_n(z)+\frac{\delta}{\delta+m}z^n
\]
Now, we have
\[
\|\Phi_n\|_\infty=\Phi_n(1)=n+1-\frac{\delta}{\delta+m}n=1+\frac{mn}{\delta+m}
\]
and
\[
\|\Phi\|_\sigma^2=\delta\left(1+\frac{m^2n}{(\delta+m)^2}\right)+\frac{\delta^2
nm}{(\delta+m)^2}=\delta\left(1+\frac{mn}{\delta+m}\right)
\]
For the orthonormal polynomial, we have
\[
\phi_n(1)=\frac{\Phi_n(1)}{\|\Phi_n\|_\sigma}
\]
For fixed $n$, we have
\[
\lim_{m\to\infty}\|\phi_n\|_\infty=\sqrt{(n+1)/\delta}
\]
\end{proof}
{\bf Remark.} This theorem has the following implication for our
original problem. Suppose we consider the class $S_\delta$ but
$\delta$ is small in $n$. Then, (\ref{scaling}) implies that
\[
M_{n,\delta_n}=\sqrt{\frac{n+1}{\delta_n}}(1+\overline{o}(1))
\]
where
\[
\delta_n=\frac{C}{nm_n}, \quad m_n\to+\infty,\quad {\rm as}\quad   n\to\infty
\]
Thus, in the small $\delta$ regime the upper bound for
$M_{n,\delta}$ is sharp. If one takes $m_n=1/n$ in the proof above
to make the total mass finite, the polynomials $\phi_n$ constructed
above are bounded in $n$ as $\delta\sim 1$.\bigskip



\section{The main theorem}\label{sec.3}

In this section, we prove the main result of this paper, the sharp
lower bound for fixed $\delta$. We start by introducing some
notation and recalling the relevant facts from the theory of
polynomials orthogonal on the unit circle. Given any polynomial
$P_n(z)=p_nz^n+\ldots+p_1z+p_0$, we can define its $n$-th reciprocal
(or the $*$--transform)
\[
P_n^*(z)=z^n \overline{P_n(1/{\overline z})}=\overline{p}_0
z^n+\overline{p}_1z^{n-1}+\ldots+\overline{p}_n
\]
Notice that if $z^*\neq 0$ is a root of $P_n(z)$, then
$(\overline{z^*})^{-1}$ is a root of $P_n^*(z)$.

The following trivial lemma will be relevant later
\begin{lemma}\label{noli}
If a polynomial $P_n$ has all zeroes inside ${\mathbb{D}}$ then
$D_n(z)=P_n(z)+P_n^*(z)$ has all zeroes on the unit circle.
\end{lemma}
\begin{proof}
We have
\[
D_n(z)=P_n^*\left(1+\frac{P_n}{P_n^*} \right), \quad z\in \mathbb{D}
\]
The first factor is zero free in $\mathbb{D}$. In the second one, $ {P_n}/{P_n^*} $ is
 a Blaschke product. If
$P_n=-P_n^*$ identically then the zeroes of $P_n$ must be on
$\mathbb{T}$ which is impossible  by assumption. Therefore, by
maximum principle,
\[
1+\frac{P_n}{P_n^*}
\]
does not have zeroes in $\mathbb{D}$. Since $D_n$ is invariant under
the $*$--transform, it has the following property: $D_n(w)=0$
implies $D_n(\overline{w}^{-1})=0$. Therefore, $D_n$ is zero free in
$|z|>1$ as well.
\end{proof}
{\bf Remark.} One can actually show that $D_n$ has the same degree as
$P_n$ under the assumptions of the lemma. Indeed, suppose not, then
we can assume that $P^*_n=z^n-1+z\widehat P(z)$ where $\widehat P$
is a polynomial. Now, we can use the argument principle: ${\rm Var}
\arg (z\widehat P)\, |_{\mathbb{T}}\geq 2\pi$ and ${\rm Var} \arg
\Bigl( z^n-(1+\epsilon)\Bigr)\, |_{\mathbb{T}}=0$ for any
$\epsilon>0$. Thus, $P^*_n-\epsilon$ has a zero in $D$ for
arbitrarily small $\epsilon>0$ which is impossible since by
assumption $P_n^*$ has all zeroes in $|z|>1$.\bigskip

We will be mostly working with the orthonormal polynomials $\phi_n$
and the corresponding $\phi_n^*$. It is well known \cite{sim1} that
all zeroes of $\phi_n$ are inside $\mathbb{D}$ thus $\phi_n^*$ has
no zeroes in $\overline{\mathbb{D}}$. However, we also need to
introduce the second kind polynomials $\psi_n$ along with the
corresponding $\psi_n^*$. Let us recall (\cite{sim1}, p. 57) that
\[
\left\{
\begin{array}{cc}
\phi_{n+1}=\rho_n^{-1}(z\phi_n-\overline{\gamma}_n\phi_{n}^*),& \phi_0=1\\
\phi_{n+1}^*=\rho_{n}^{-1}(\phi_n^*-\gamma_n z\phi_n),& \phi^*_0=1
\end{array}
\right.
\]
and the second kind polynomials satisfy the recursion with Schur
parameters $-\gamma_n$, i.e.
\begin{equation}\label{secon}
\left\{
\begin{array}{cc}
\psi_{n+1}=\rho_n^{-1}(z\psi_n+\overline{\gamma}_n\psi_{n}^*),& \psi_0=1\\
\psi_{n+1}^*=\rho_{n}^{-1}(\psi_n^*+\gamma_n z\psi_n),& \psi^*_0=1
\end{array}
\right.
\end{equation}
and the following Bernstein-Szeg\H{o} approximation result is valid:
\begin{theorem}Suppose $d\mu$ is a probability measure and $\{\phi_j\}$ and $\{\psi_j\}$
are the
corresponding orthonormal polynomials of the first/second kind,
respectively.  Then,  for any $N$, the function
\[
F_N(z)=\frac{\psi_N^*(z)}{\phi_N^*(z)}=\int_{\mathbb{T}}
C(z,e^{i\theta})d\mu_N(\theta), \quad d\mu_N(\theta)=\frac{1}{2\pi
|\phi_N(e^{i\theta})|^2}=\frac{1}{2\pi |\phi^*_N(e^{i\theta})|^2}
\]
has the first $N$ Taylor coefficients (which are the moments of
$d\mu_N$ up to a constant) identical to the Taylor coefficients of
the function
\[
F(z)=\int_{\mathbb{T}} C(z,e^{i\theta})d\mu(\theta)
\]
In particular, the polynomials $\{\phi_j\}$ and $\{\psi_j\}$, $j\leq
N$ are the orthonormal polynomials of the first/second kind for the
measure $d\mu_N$.
\end{theorem}
In a way, the converse to the Bernstein-Szeg\H{o} approximation
holds as well.
\begin{lemma}\label{vspomag}
The polynomial $P_n(z)$ of degree $n$ is the orthonormal polynomial
for the probability measure with infinitely many growth points if
and only if
\begin{itemize}
\item[1.] $P_n(z)$ has all $n$ zeroes inside $\mathbb{D}$ (counting the multiplicities).
\item[2.] The normalization condition
\[
\int_\mathbb{T} \frac{d\theta}{2\pi|P_n(e^{i\theta})|^2}=1
\]
is satisfied.
\end{itemize}
\end{lemma}
Now, we are ready to formulate the main result of the
paper.
\begin{theorem}\label{T3} For any $\delta\in (0,\delta_0)$ with $\delta_0$ sufficiently small,
 we have
\begin{equation}\label{osnova}
M_{n,\delta} > C(\delta)\sqrt{{n}}
\end{equation}
\end{theorem}

{\bf Remark.} We are not trying to track the dependence of
$C(\delta)$ on $\delta$ as $\delta\to 0$ and we do not attempt to make $\delta_0=1$.

\begin{proof}
We first need to rewrite the problem in more convenient form.

\begin{lemma}
To prove (\ref{osnova}), it is sufficient to find a polynomial  $\phi_n^*$ and a Herglotz
function $\widetilde F$ which satisfy the following properties
\begin{itemize}
\item[1.] $\phi_n^*(z)$ has no roots in $\mathbb{D}$.

\item[2.] Normalization
\begin{equation}\label{norma}
\int_{\mathbb{T}} |\phi_n^*(z)|^{-2}d\theta =2\pi
\end{equation}

\item[3.] Large uniform norm, i.e.
\[
|\phi^*_n(1)|\sim \sqrt{{n}}
\]

\item[4.]  $\widetilde F\in C^{\infty}(\mathbb{T})$,\, $\Re \widetilde F>0$ on $\mathbb{T}$,
and
\begin{equation}\label{norka}
\frac{1}{2\pi}\int_{\mathbb{T}}\Re \widetilde F(e^{i\theta}) d\theta  =1
\end{equation}
Moreover,
\begin{equation}
|\phi^*_n(z)|+|\widetilde F(z)(\phi_n(z)-\phi_n^*(z))|<C_1(\delta) \Bigl(\Re
\widetilde F(z)\Bigr)^{1/2} \label{main}
\end{equation}
uniformly in $z\in \mathbb{T}$.
\end{itemize}
\end{lemma}
\begin{proof}
By lemma \ref{vspomag}, the first two conditions guarantee that $\phi_n(z)$ is orthonormal
polynomial of some probability measure and it also determines its first $n$
Schur parameters: $\gamma_0,\ldots, \gamma_{n-1}$. The third gives the necessary growth. Next,
 let
us show that the fourth condition is sufficient for the existence of
a measure $\sigma\in S_\delta$ for which $\phi_n$ is the $n$--th  orthonormal
polynomial.

Notice that $\widetilde F$ defines the corresponding probability measure $\widetilde \sigma$
 which is purely absolutely continuous and has positive smooth density $\widetilde\sigma'$ given
  by
\begin{equation}\label{fefl}
\widetilde\sigma'(\theta)=\frac{\Re\widetilde F(e^{i\theta})}{2\pi}
\end{equation}
Denote its Schur parameters by $\{\widetilde \gamma_j\}, \, j=0,1,\ldots$ and the orthonormal
 polynomials of the first and second kind by $\{\widetilde \phi_j\}$ and
 $\{\widetilde \psi_j\}, \, j=0,1,\ldots$, respectively.  By Baxter's theorem \cite{sim1}
 we have $\widetilde\gamma_j\in \ell^1$ (in fact, the decay is much stronger but $\ell^1$
 is enough for our purposes). Then, let us consider the measure $\sigma$  which has the
  following Schur parameters
\[
\gamma_0,\ldots,\gamma_{n-1},\widetilde \gamma_0, \widetilde \gamma_1,\ldots
\]
We will show that this measure satisfies the Steklov's condition.
Denote
\begin{equation}\label{otoj}
\gamma_n=\widetilde \gamma_0, \gamma_{n+1}=\widetilde \gamma_1, \ldots
\end{equation}
The Baxter's theorem yields that $\sigma$ is purely a.c., $\sigma'$ belongs to the Wiener's
 class $W(\mathbb{T})$, and $\sigma'$ is positive on $\mathbb{T}$.
The first $n$ orthonormal polynomials corresponding to the measure $\sigma$ will be
$\{\phi_j\}, \, j=0,\ldots, n-1$.
Let us compute the polynomials $\phi_j$ and $\psi_j$ (orthonormal with respect to $\sigma$)
 for the indexes $j>n$.
Since the second kind polynomials correspond to Schur parameters $\{-\gamma_j\}$ (
see (\ref{secon})), the recursion can be rewritten in the following matrix form
\begin{equation}\label{m-ca}
\left(
\begin{array}{cc}
\phi_{n+m} & \psi_{n+m}\\
\phi_{n+m}^* & -\psi_{n+m}^*
\end{array}
\right)=\left(
\begin{array}{cc}
A_m & B_m\\
C_m & D_m
\end{array}
\right)\left(
\begin{array}{cc}
\phi_{n} & \psi_{n}\\
\phi_{n}^* & -\psi_{n}^*
\end{array}
\right)
\end{equation}
where $A_m, B_m, C_m, D_m$ satisfy
\begin{eqnarray*}
\left(
\begin{array}{cc}
A_0 & B_0\\
C_0 & D_0
\end{array}
\right)=\left(
\begin{array}{cc}
1 & 0\\
0 & 1
\end{array}
\right),\hspace{6cm} \\ \left(
\begin{array}{cc}
A_m & B_m\\
C_m & D_m
\end{array}
\right)=\frac{1}{\widetilde \rho_0\cdot\ldots\cdot\widetilde\rho_{m-1}}\left(
\begin{array}{cc}
z & -\widetilde\gamma_{m-1}\\
-z\widetilde\gamma_{m-1} & 1
\end{array}
\right)\cdot\ldots\cdot\left(
\begin{array}{cc}
z & -\widetilde\gamma_0\\
-z\widetilde\gamma_0 & 1
\end{array}
\right)
\end{eqnarray*}
and thus depend only on $\gamma_{n}, \ldots, \gamma_{n+m-1}$
(i.e., $\widetilde\gamma_{0}, \ldots, \widetilde\gamma_{m-1}$ by (\ref{otoj})).
Moreover, we have
\[
\left(
\begin{array}{cc}
\widetilde\phi_m & \widetilde\psi_m\\
\widetilde\phi_m^* & -\widetilde\psi^*_m
\end{array}
\right)= \left(
\begin{array}{cc}
A_m & B_m\\
C_m & D_m
\end{array}
\right)\left(
\begin{array}{cc}
1 & 1\\
1 & -1
\end{array}
\right)
\]
Thus,
$
A_m=(\widetilde \phi_m+\widetilde \psi_m)/2, B_m=(\widetilde\phi_m-\widetilde \psi_m)/2,C_m=(\widetilde\phi^*_m-\widetilde\psi^*_m)/2, D_m=(\widetilde\phi^*_m+\widetilde\psi^*_m)/2
$
and substitution into (\ref{m-ca}) yields
\begin{equation}\label{intert}
2\phi_{n+m}^*=\phi_n(\widetilde\phi_m^*-\widetilde\psi^*_m)+\phi_n^*(\widetilde\phi_m^*+\widetilde\psi^*_m)=
\widetilde\phi_m^*\left(\phi_n+\phi_n^*+\widetilde
F_m(\phi_n^*-\phi_n)\right)
\end{equation}
where
\[
\widetilde F_m(z)=\frac{\widetilde\psi^*_m(z)}{\widetilde\phi^*_m(z)}
\]
Since $\{\widetilde\gamma_n\}\in \ell^1$ and $\{\gamma_n\}\in \ell^1$, we have
(\cite{sim1}, p.225)
\[
\widetilde F_m\to \widetilde F \,\, {\rm as\,\,}  m\to\infty\,\,\,{\rm and\,\,}
\phi_n^*\to \Pi, \,\,\widetilde \phi_n^*\to\widetilde\Pi \,\,{\rm as}\, n\to\infty
\]
uniformly on $\overline{\mathbb{D}}$. The functions $\Pi$ and $\widetilde\Pi$
are the Szeg\H{o} functions of $\sigma$ and $\widetilde\sigma$, respectively,
i.e. they are outer functions in $\mathbb{D}$ that give the factorization
\begin{equation}\label{facti}
|\Pi|^{-2}=2\pi\sigma', \quad |\widetilde \Pi|^{-2}=2\pi\widetilde \sigma'
\end{equation}
In (\ref{intert}), send $m\to\infty$ to get
\begin{equation}\label{facti1}
2\Pi= \widetilde\Pi\left(\phi_n+\phi_n^*+\widetilde
F(\phi_n^*-\phi_n)\right)
\end{equation}
Thus, the first formula in (\ref{facti}) shows that in the category of sufficiently
regular measures the
Steklov's condition $\sigma'>\delta/(2\pi)$ is equivalent to
\begin{equation}\label{oc-1}
\left|\widetilde \Pi\left(\phi_n+\phi_n^*+\widetilde
F(\phi_n^*-\phi_n)\right)\right|\leq \frac{2}{\sqrt{\delta}}, \quad z=e^{i\theta}\in \mathbb{T}
\end{equation}
Since $|\phi_n|=|\phi_n^*|$ on $\mathbb{T}$, we have
\[
\left|\widetilde \Pi\left(\phi_n+\phi_n^*+\widetilde
F(\phi_n^*-\phi_n)\right)\right|\leq 2|\widetilde \Pi|\left(|\phi^*_n|+|\widetilde F(\phi^*_n-\phi_n)|\right)<2C_1(\delta) |\widetilde \Pi|\left(\Re \widetilde F\right)^{1/2}=2C_1(\delta)
\]
due to (\ref{main}), (\ref{fefl}), and the second formula in (\ref{facti}). Thus,
to guarantee (\ref{oc-1}), we only need to take $C_1(\delta)=\delta^{-1/2}$ in (\ref{main}).
 In this paper, we assume $\delta$ to be fixed so the exact formulas for $C(\delta)$
 and $C_1(\delta)$ will not be needed.
\bigskip
\end{proof}
Having rewritten the Steklov's condition, we only need to present the
polynomial $\phi_n^*$ and the function $\widetilde F$ which satisfy the
conditions given in the lemma.
Take $\epsilon_n=n^{-1}$.

{\bf 1.  Choice of $
\widetilde F$.} Consider two parameters: $\alpha\in (1/2,1)$ and $\rho\in (0,\rho_0)$ where $\rho_0$ is sufficiently small. Take
\[
\widetilde F(z)=\widetilde C_n\left(\rho(1+\epsilon_n-z)^{-1}+(1+\epsilon_n - z)^{-\alpha}\right),
\]
where the positive constant $\widetilde C_n$ will be chosen later.
Clearly $\widetilde F$ is smooth and has a positive real part in $\overline{\mathbb{D}}$.
Notice that for $z=e^{i\theta}\in \mathbb{T}$ and $\theta\sim 0$ we
have
\[
1+\epsilon_n-z=\epsilon_n+\frac{\theta^2}{2}-i\theta+O(\theta^3)
\]
and so
\begin{equation}\label{herglotz}
\widetilde
F(e^{i\theta})=\widetilde C_n\left(\frac{\rho\epsilon_n}{\epsilon_n^2+\theta^2}+i\frac{\rho\theta}{\epsilon_n^2+\theta^2}+(1+\epsilon_n-z)^{-\alpha}+O\left(\frac{|\theta|}{\epsilon_n+|\theta|}\right)\right)
\end{equation}
For small fixed $\upsilon$ and $\theta\in (-\upsilon,\upsilon)$ we have
\begin{equation}\label{j1}
(1+\epsilon_n-e^{i\theta})^{-\alpha}\sim (\epsilon_n^2+\theta^2)^{-\alpha/2}\exp(-i\alpha \Gamma_n(\theta))
\end{equation}
\begin{equation}\label{j2}
\Gamma_n(\theta)=-\arctan \left( \frac{\sin\theta}{1+\epsilon_n-\cos\theta}\right)\in (-\pi/2, \pi/2)
\end{equation}
The following bound is true
\[
\left\|\frac{\epsilon_n}{\epsilon_n^2+\theta^2}+|(1+\epsilon_n-z)^{-\alpha}|+O\left(\frac{|\theta|}{\epsilon_n+|\theta|}\right)\right\|_{L^1[-\pi,\pi]}\sim C_2(\alpha,\rho)
\]
uniformly in $n$.  Then, we choose $\widetilde C_n$ to guarantee (\ref{norka}) and then $\widetilde C_n\sim C_2(\alpha,\rho)$ uniformly in $n$.

For $|\theta|<\epsilon_n$
\begin{equation}\label{tit1}
|\widetilde F|\sim \epsilon_n^{-1}= n
\end{equation}
and for $|\theta|>\epsilon_n$
\begin{equation}\label{tit2}
|\widetilde F|\sim |
\theta|^{-1}
\end{equation}
\bigskip

{\bf 2. Choice of $\phi_n^*$.} Let $\phi_n^*$ be as follows
\[
\phi_n^*(z)=C_n f_n(z), \, f_n(z)=P_m(z)+Q_m(z)+Q_m^*(z)
\]
where $P_m$ and $Q_m$ are certain polynomials of degree
\begin{equation}\label{propa}
m=[\delta_1
n]
\end{equation}
 where $\delta_1$ is small and will be chosen later.  Notice here that $Q_m^*$ is defined by applying the $n$--th order star operation. The constant $C_n$ will be chosen in such a way that
\[
\int_{-\pi}^\pi |\phi^*_n|^{-2}d\theta=2\pi
\]
(i.e. (\ref{norma}) is satisfied). To prove the theorem, we only need to show that
\begin{equation}\label{si-n}
C_n=\left(\int_{-\pi}^\pi |f_n|^{-2}d\theta\right)^{1/2}\sim 1
\end{equation}
uniformly in $n$ and that $f_n$ satisfies the other conditions of the lemma.
\bigskip

 Consider the Fejer kernel
\begin{equation}\label{feya}
\cal{F}_m(\theta)=\frac{1}{m}\frac{\sin^2(m\theta/2)}{\sin^2(\theta/2)}=\frac{1}{m}\frac{1-\cos(m\theta)}{1-\cos(\theta)}, \quad \cal{F}_m(0)=m
\end{equation}
and the Taylor approximation to the function
$
(1-z)^{-\alpha}
$,
i.e.
\[
R_{(k,\alpha)}(z)=c_0+\sum_{j=1}^k c_jz^j
\]
(see Appendix B for the detailed discussion).
We define $Q_m$ as an analytic polynomial without zeroes in
$\mathbb{D}$ which gives Fejer-Riesz factorization
\begin{equation}\label{fact}
|Q_m(z)|^2=\cal{G}_m(\theta)+|R_{(m,\alpha/2)}(e^{i\theta})|^2
\end{equation}
\begin{equation}\label{sdvig}
\cal{G}_m(\theta)=\cal{F}_m(\theta)+\frac
12\cal{F}_m\left(\theta-\frac{\pi}{m}\right)+\frac
12\cal{F}_m\left(\theta+\frac{\pi}{m}\right)
\end{equation}
Clearly, the right hand side of (\ref{fact}) is positive
trigonometric polynomial of degree $m$ so this factorization is
possible and $Q_m$ is unique up to a unimodular factor. We choose this factor in such a way that $Q_m(0)>0$, i.e.
\begin{equation}\label{mult-mult}
Q_m(z)=\exp\left(\frac{1}{2\pi} \int C(z,e^{i\theta})\log |Q_m(e^{i\theta})|d\theta\right)
\end{equation}
Notice that $Q_m^*$ is a polynomial of degree $n$ with positive leading coefficient. Since $|Q_m(e^{i\theta})|$ is even in $\theta$, the multiplicative representation shows that $H(z)=\log Q_m(z)$ is analytic in $\mathbb{D}$ and has real Taylor coefficients (indeed, $\overline{H(z)}=H(\overline{z})$). That, on the other hand, implies that $Q_m(z)=e^{H(z)}$ has real coefficients as well.

For $P_m$, we take
\begin{equation}\label{ppp}
P_m(z)=Q_m(z)(1-z)(1-0.1R_{(m,-(1-\alpha))}(z))
\end{equation}
and $\deg P_m=2m+1<n$ by the choice of small $\delta_1$. Consequently, $\deg \phi^*_n=n$.\bigskip

Now that we have chosen $\widetilde F$ and $\phi_n^*$, it is left to show that they satisfy the conditions of the lemma.

{\bf 1. $f_n$ has no zeroes in $\overline{\mathbb{D}}$}.
For $f_n$, we can write
\begin{equation}\label{efn}
f_n=Q_m\left(
(1-z)(1-0.1R_{(m,-(1-\alpha))}(z))+1+z^ne^{-2i\phi}
\right), \quad z\in \mathbb{T}
\end{equation}
where
\[
e^{-2i\phi}=\frac{\overline{Q_m}}{Q_m}
\]
so $\phi$ is an argument of $Q_m$.
The polynomial $Q_m$ has no zeroes in $\mathbb{D}$ and
\begin{equation}\label{dunduk}
(1-z)(1-0.1R_{(m,-(1-\alpha))}(z))+1+\frac{Q_m^*}{Q_m}
\end{equation}
is analytic in $D$ and has positive real part. Indeed,
\[
\Re \left( 1+\frac{Q_m^*}{Q_m}\right)\geq 0, \quad z\in \mathbb{T}
\]
since $|Q_m|=|Q^*_m|$.

Since $Q_m$ has real coefficients, $Q_m(1)$ is real and so $\phi(0)=0$ and
\[
\Re \left( 1+\frac{Q_m^*}{Q_m}\right)=2, \quad z=1
\]
 For the first term in (\ref{dunduk}), we have
\begin{equation}\label{parts}
\Re(1-z)(1-0.1R_{(m,-(1-\alpha))}(z))=(1-\cos\theta)(1-0.1X)-0.1Y\sin\theta
\end{equation}
where
\[
X=\Re R_{(m,-(1-\alpha))},\quad Y=\Im R_{(m,-(1-\alpha))}
\]
Notice that $|R_{(m,-(1-\alpha))}|<3$ for $z\in \mathbb{D}$ so $|X|<3$ in $\mathbb{D}$ as well. The function $Y$ is odd in $\theta$ and $Y(\theta)<0$ for $\theta>0$ (as follows from the lemma \ref{poly1} in Appendix A).
Thus, the function $f_n(z)$ has a positive real part on $\mathbb{T}$ and, in particular,
has no zeroes in $\mathbb{D}$. We conclude then that $\phi_n^*$ has
no zeroes in $\mathbb{D}$.\smallskip

{\bf Remark.} The polynomial $Q_m+Q_m^*$ has all zeroes on the unit
circle as follows from lemma \ref{noli}. The relatively small correction $P_m$ moves them away from
$\overline{\mathbb{D}}$.\bigskip

{\bf 2. The growth at $z=1$.}  The
(\ref{efn}) implies
\[
f_n(1)=2Q_m(1)
\]
Then, (\ref{feya}) and (\ref{fact}) yield
\[
|f_n(1)|\gtrsim \sqrt m\sim \sqrt n
\]
due to (\ref{propa}).\bigskip

{\bf 3. Steklov's condition.}
We need to check (\ref{main}) with $\phi_n$ replaced by $f_n$.  From  (\ref{fact}) and (\ref{efn}), we get
\[
|f_n|^2\lesssim
|Q_m|^2=\cal{G}_m(\theta)+|R_{(m,\alpha/2)}(e^{i\theta})|^2
\]
From  lemma \ref{poly2}
\[
|R_{(m,\alpha/2)}(e^{i\theta})|^2\lesssim (\epsilon_m+|\theta|)^{-\alpha}\lesssim  (\epsilon_n+|\theta|)^{-\alpha}
\]
The exact form of the Fejer's kernel (\ref{feya}) gives
\[
\cal{G}_m(\theta)\lesssim  \frac{m}{m^2\theta^2+1} \lesssim  \frac{n}{n^2\theta^2+1}=\frac{\epsilon_n}{\epsilon_n^2+\theta^2}
\]
The bounds (\ref{herglotz}), (\ref{j1}), and (\ref{j2}) give
\[
\Re\widetilde F\sim \frac{\epsilon_n}{\epsilon_n^2+\theta^2}+(\epsilon_n^2+\theta^2)^{-\alpha/2}
\]
Therefore,
\begin{equation}\label{vtoro}
|f_n|\lesssim |Q_m|\lesssim  (\Re \widetilde F)^{1/2}
\end{equation}
Then, for the second term in (\ref{main}), we get
\[
|\widetilde F(f_n-f_n^*)|^2=
|\widetilde F(P_m-P_m^*)|^2\lesssim |\widetilde
F(1-z)|^2|Q_m|^2
\]
The uniform bounds
\[
|\widetilde F(1-z)|\lesssim 1,\,
|Q_m|^2\lesssim \Re \widetilde F
\]
together with (\ref{vtoro}), imply (\ref{main}) with $\phi_n$ replaced by $f_n$.
\bigskip

{\bf 4. Normalization.}
We only need to check now that
\[
\int_{\mathbb{T}}\frac{1}{|f_n|^2}d\theta\sim 1
\]
(see (\ref{si-n})). We only need to consider $\theta\in [0,
\upsilon]$ where $\upsilon$ is small and fixed as $
f_n(\theta)$ is even and
\[
|f_n|>C_5(\upsilon), \quad \theta\in [\upsilon,\pi]
\]
as follows from   (\ref{fact}), (\ref{efn}),  and (\ref{parts}).

Notice that (\ref{fact}) yields
\[
|Q_m(e^{i\theta})|^2\geq |R_{(m,\alpha/2)}(e^{i\theta})|^2
\]
and the lemma \ref{poly2}  from Appendix A
 yields
\[
|f_n|^2\gtrsim (m^{-1}+|\theta|)^{-\alpha}\left|
(1-z)(1-0.1R_{(m,-(1-\alpha)}(z))+1+z^ne^{2i\phi}\right|^2
\]
Then, the representation (\ref{parts}) gives
\begin{equation}\label{xyz}
\left|
(1-z)(1-0.1R_{(m,-(1-\alpha)}(z))+1+z^ne^{2i\phi}\right|^2=
\end{equation}
\[
\Bigl((1-\cos\theta)(1-0.1X)-0.1Y\sin\theta +1+\cos(n\theta-2\phi)\Bigr)^2+
\]
\[
\Bigl(  -0.1Y(1-\cos\theta)-(1-0.1X)\sin\theta+\sin(n\theta-2\phi)\Bigr)^2
\]
For the first term, we have
\[
(1-\cos\theta)(1-0.1X)-0.1Y\sin\theta +1+\cos(n\theta-2\phi)\geq -0.1Y\sin\theta\gtrsim \theta^{2-\alpha}
\]
where the last inequality follows from lemma \ref{poly1}. Thus,
\begin{equation}\label{inv}
|f_n|^2\gtrsim (m^{-1}+|\theta|)^{-\alpha}\left[\Bigl(
(\theta^{2-\alpha}) ^2 + \Bigl(
\Psi(\theta)+\sin(n\theta-2\phi)\Bigr)^2 \right]
\end{equation}
where
\[
\Psi= -0.1Y(1-\cos\theta)-(1-0.1X)\sin\theta
\]
The lemma \ref{der-der} gives
\[
|\Psi'(\theta)|\lesssim 1
\]
uniformly in $n$ and $\Psi(0)=0$ implies
\[
|\Psi(\theta)|\lesssim \theta
\]
by integration.

For the phase $\phi$, we have $\phi(0)=0$  and
\[
|\phi'(\theta)|\lesssim m
\]
as proved in Appendix B. Since we have the derivative of $\phi$ under control, making
$\delta_1$ small we can make sure that the function
$
n\theta-2\phi(\theta)
$
is monotonically increasing and
\[
n/2<(n\theta-2\phi(\theta))'<2n
\]
Denote the consecutive zeroes of
\[ \Psi(\theta)+\sin(n\theta-2\phi(\theta))\]
on the interval $[0,\upsilon]$
by $\{\theta_j\}$: $0=\theta_0<\theta_1<\ldots<\theta_k\leq \upsilon$. These zeroes can be found from two sets of equations
\[
F_1(\theta_{2j})= 2j\pi, \quad j=0,1,\dots,N_1; \quad F_1(\theta)=n\theta-2\phi(\theta)+\arcsin \Psi(\theta)
\]
and
\[
F_2(\theta_{2j-1})=(2j+1)\pi, \quad j=0,1,\dots,N_2; \quad F_2(\theta)=n\theta-2\phi(\theta)-\arcsin \Psi(\theta)
\]
Since
\[
n/2<F'_{1(2)}<2n \quad{\rm and }\quad F_{1(2)}(0)=0
\]
these $\{\theta_j\}$ exist, $N_{1(2)}\sim n$, and
\begin{equation}\label{porovnu}
\frac{1}{20 n}<|\theta_{j+1}-\theta_j|<\frac{20}{n}, \quad \theta_j\sim \frac{j}{n}
\end{equation}

Now, consider the intervals $I_0=[0,0.01n^{-1}]$, $I_j=[\theta_j-0.01n^{-1}, \theta_j+0.01n^{-1}], j=1,\ldots, k$ centered around $\{\theta_j\}$. For $\theta\in I_j$, we have
\[
|\Psi(\theta)+\sin(n\theta-2\phi(\theta))|=\Bigl|\int_{\theta_j}^\theta  \Bigl(\Psi'(\xi)+\cos(n\xi-2\phi(\xi))(n-2\phi'(\xi)\Bigr)d\xi\Bigr|\sim n|\theta-\theta_j|
\]
and
\[
|\Psi(\theta)+\sin(n\theta-2\phi(\theta))|\sim 1
\]
outside $\cup_{j} I_j$.

Therefore, from (\ref{inv}), we have
\[
\|f_n^{-1}\|_2^2\lesssim \int\limits_{[0,\upsilon]\cap (\cup_j
I_j)^c} \theta^{\alpha}d\theta+\sum_{j=1}^{Cn}
\theta_j^{\alpha}\int_0^{Cn^{-1}}
\frac{d\theta}{(\theta_j^{2-\alpha})^2+n^2\theta^2}
\]
\[
\lesssim 1+\frac 1n \sum_{j=1}^{Cn} \theta_j^{-2+2\alpha}\sim 1+\int_0^1
\theta^{2\alpha-2}d\theta\lesssim 1
\]
where we used (\ref{porovnu}) and  $\alpha\in (1/2,1)$.
\end{proof}
{\bf Remark.} Although the limitation $\alpha>1/2$ seems rather artificial, we believe it is the range of $\alpha\in (1-\epsilon,1)$ with small $\epsilon$ that is essential for the argument to hold.
\vspace{0.5cm}

Our method allows to compute the measure of orthogonality $\sigma$ for which the orthonormal
polynomial has the required size and it is interesting to compare it to the results on the maximizers we obtained before.
This $\sigma$ is purely absolutely continuous. To find it we use (\ref{facti}) and (\ref{facti1}) and get
\[
\sigma'=\frac{4\widetilde\sigma'}{|\phi_n+\phi_n^*+\widetilde F(\phi_n^*-\phi_n)|^2}=\frac{2\Re\widetilde F}{\pi|\phi_n+\phi_n^*+\widetilde F(\phi_n^*-\phi_n)|^2}
\]
This expression is explicit as we know the formulas for all functions involved. We have
\begin{equation}\label{chto}
\sigma^{-1}=\frac{C_n |Q_n|^2}{\Re \widetilde F}\cdot\left| 2(1+e^{i(n\theta-2\phi)})+H_n(e^{i\theta})+e^{i(n\theta-2\phi)}\overline{H_n(e^{i\theta})}+\widetilde F \left(-H_n(e^{i\theta})+e^{i(n\theta-2\phi)}\overline{H_n(e^{i\theta})}\right)  \right|^2
\end{equation}
where
\[
H_n(z)=(1-z)(1-0.1R_{(m,-(1-\alpha))}(z)), \quad C_n\sim 1
\]
as follows from (\ref{ppp}).

Consider the first factor. We can apply (\ref{tit1}), (\ref{tit2}), and lemma \ref{poly2} to get
\[
\frac{|Q_n|^2}{\Re \widetilde F}\sim 1
\]
Recall that $\widetilde F$ is given by
\[
\widetilde F(z)=\widetilde C_n(\rho(1+\epsilon_n-z)^{-1}+(1+\epsilon_n-z)^{-\alpha}), \, \widetilde C_n=(\rho/(1+\epsilon_n)+(1+\epsilon_n)^{-\alpha})^{-1}
\]
where the choice for constant $\widetilde C_n$ comes from the normalization (\ref{norka}).

Substitution into the second factor in (\ref{chto}) gives
\[
(\sigma')^{-1}\sim \left| (2+\overline{H}_n(1+\widetilde F))\left(e^{i(n\theta-2\phi)}+\frac{2+H_n(1-\widetilde F)}{2+\overline H_n (1+\widetilde F)}\right)\right|^2
\]
For $z=e^{i\theta}$ and small positive $\theta$ (negative values can be handled similarly), we have
\[
H_n(e^{i\theta})=-i\theta+0.1\theta i R_{(m,-(1-\alpha))}(e^{i\theta})+O(|\theta|^{2})
\]
and for $R_{(m,-(1-\alpha))}$ lemma \ref{poly1} can be used.
For $\widetilde F$,
\[
\widetilde F(e^{i\theta})=\widetilde C_n\left(  \frac{\rho}{\epsilon_n-i\theta}+\frac{1}{(\epsilon_n-i\theta)^\alpha}
\right)+O(1)
\]
and therefore the first factor
\[
|2+\overline{H}_n(1+\widetilde F)|\sim 1
\]
for $\rho\in(0,\rho_0)$ as long as $\rho_0$ is small.

Consider
\[
J=\frac{2+H_n(1-\widetilde F)}{2+\overline H_n (1+\widetilde F)}
\]
Notice first that $|J|<1$ for $\theta\neq 0$. Indeed,
\begin{eqnarray*}
1-|J|^2=\frac{-|2+H_n(1-\widetilde F)|^2+|2+\overline H_n (1+\widetilde F)|^2}{|2+\overline H_n (1+\widetilde F)|^2}=\frac{4\Re\left((H_n+\overline{H}_n+|H_n|^2)\overline{\widetilde F}\right)}{{|2+\overline H_n (1+\widetilde F)|^2}}=
\\
=\frac{4(\Re \widetilde F)(|H_n|^2+2\Re H_n)}{{|2+\overline H_n (1+\widetilde F)|^2}}
\end{eqnarray*}
and the last expression is positive by the choice of $\widetilde F$ and $H_n$. For $\theta=0$, we have $J(0)=1$.

For small $\theta$, we have asymptotics
\begin{equation}\label{masa1}
1-|J|^2\sim |\theta|^{2-\alpha} \left(\frac{\rho \epsilon_n}{\epsilon_n^2+\theta^2}+|\theta|^{-\alpha}\right), \quad |\theta|>0.1n^{-1}
\end{equation}
and
\[
1-|J|^2\sim \left(  \frac{\rho\epsilon_n}{\epsilon_n^2+\theta^2}+\epsilon_n^{-\alpha}  \right)|\theta|n^{\alpha-1}, \quad |\theta|<0.1n^{-1}
\]
If we write
\[
J^{-1}=r(\theta)e^{i\Upsilon(\theta)}
\]
then
\[
(\sigma')^{-1}\sim  |e^{i(n\theta-2\phi(\theta)+\Upsilon(\theta))}+r(\theta)|^2=(\cos(n\theta-2\phi(\theta)+\Upsilon(\theta))+r(\theta))^2+\sin^2 (n\theta-2\phi(\theta)+\Upsilon(\theta))
\]

Consider solutions $\{\widehat \theta_j\}$ to
\[
n\theta-2\phi(\theta)+\Upsilon(\theta)=\pi+2\pi j, \quad |j|<N_3
\]
that belong to some small fixed arc $|\theta|<\upsilon$. We have $\Upsilon(0)=0$ and the direct estimation gives
\[
|\Upsilon'(\theta)|<0.1n
\]
uniformly for all $\theta$. Indeed, it is sufficient to prove
\[
\left|\partial_\theta\left(\frac{2+H_n(1-\widetilde F)}{2+\overline H_n(1-\overline{\widetilde F})}\right)\right|<0.01n
\]
and
\[
\left|\partial_\theta\left(\frac{2+\overline H_n(1+\widetilde F)}{2+ H_n(1+\overline{\widetilde F})}\right)\right|<0.01n
\]
The both estimates follow from the bounds in lemmas \ref{poly1} and \ref{der-der} by the choice of small $\rho$
(one can actually obtain much better bounds by looking at $\partial_\theta J$ itself).

Now we can argue that the distance between the consecutive $\{\widehat \theta_j\}$ is $\sim n^{-1}$ and the density $\sigma'$ has  spikes of width $\sim n^{-1}$ around these points each carrying the weight $m_j$ where
\[
m_j\lesssim
\int_0^{Cn^{-1}}\frac{d\theta}{(r(\theta_j)-1)^2+n^2\theta^2}\lesssim \frac{1}{n(r(\theta_j)-1)}
\]
and
\[
\sum_{j=-N_3}^{N_3} m_j\lesssim \int_{0.1n^{-1}<|\theta|<\upsilon}
\frac{d\theta}{r(\theta)-1}<\infty
\]
as follows from (\ref{masa1}) and $\alpha\in (1/2,1)$. Away from these spikes the density is $\sim 1$.
In this argument we assumed that $\rho_0$ is sufficiently small.

\vspace{1cm}

\section{ The polynomial entropies and the Steklov
class}

In recent years, a lot of efforts were made (see, e.g., \cite{ap1, ap2, ap3}) to study the  so-called polynomial entropy
\[
\int_\mathbb{T} |\phi_n|^2\log|\phi_n|d\sigma
\]
where $\phi_n$ are orthonormal with respect to $\sigma$.
Since $\sup_{x\in [0,1]} x^2|\log x|<\infty$, this quantity is bounded if and only if
\[
 \int_\mathbb{T}
|\phi_n|^2\log^+|\phi_n|d\sigma
\]
is bounded. The last expression is important as it contains the information on the size of $\phi_n$.
In this section, we consider the following variational problem
\[
\Omega_n(\cal K)=\sup_{\sigma\in \cal{K}} \int_\mathbb{T}
|\phi_n|^2\log^+|\phi_n|d\sigma
\]
where $\phi_n$ is the $n$-th orthonormal polynomial with respect to
$\sigma$ taken in $\cal{K}$, some special class of measures. It is an interesting question to describe those $\cal{K}$ for which $\Omega_n(\cal{K})$ is bounded in $n$. So far, this is known only for very few $\cal{K}$, e.g., the Baxter class of measure. For the Szeg\H{o} class with measures normalized by the $\ell^2$ norm of Schur parameters, the sharp estimate $\Omega_n\sim \sqrt n$ is known \cite{dk}.
In this section, we will obtain the sharp bound on $\Omega_n(S_\delta)$.
\begin{lemma}
If $\delta\in (0,\delta_0)$ and $\delta_0$ is sufficiently small, then
\[
\Omega_n(S_\delta)\sim \log n
\]
\end{lemma}
\begin{proof}
If one takes the measure $\sigma$ and the polynomial $\phi_n$
constructed in the previous section, then
\[
\Omega_n(S_\delta)\gtrsim
\int_{\mathbb{T}}|\phi_n|^2\log^+|\phi_n|d\theta\gtrsim 1+
\int_{0.01n^{-1}}^{\theta_1-0.01n^{-1}} |f_n|^2\log^+|f_n|d\theta
\]
where $\theta_1$ was introduced in the proof of the main theorem. On that interval,
$|\Psi(\theta)+\sin(n\theta-2\phi)|>C$ and so $|f_n|\sim |Q_m|$. That follows from (\ref{efn}) and the verification of the normalization condition in the proof of the main theorem. Then, the estimate (\ref{uh-uh}) gives a very rough lower bound
\[
|Q_m(e^{i\theta})|^2\gtrsim  \frac{m}{m^2\theta^2+1}+1
\]
This shows $|Q_m|\sim \sqrt n$ on the interval $(0.01n^{-1},\theta_1-0.01n^{-1})$ and so
\[
\Omega_n(S_\delta)\gtrsim n^{-1}\cdot n\cdot \log n\sim \log n
\]
Therefore the polynomial entropy grows at least as the logarithm. On the other hand, the trivial upper bound
\[
\|\phi_n\|_\infty\lesssim \sqrt{n}
\]
implies that
\[
\Omega_n(S_\delta)\lesssim \log n
\]
\end{proof}

 {\bf Acknowledgement.} The research of S.D. was supported by
NSF grant DMS-1067413. The research of A.A. and D.T. was supported
by the grants RFBR 13-01-12430 OFIm and  RFBR 11-01-00245 and
Program 1 DMS RAS. The hospitality of IMB (Bordeaux) and IHES
(Paris) is gratefully acknowledged by S.D. The authors thank Stas
Kupin and Fedor Nazarov for interesting comments.

\vspace{1cm}

\section{Appendix A}

In this Appendix, we will introduce and study the polynomials that
approximate the function $(1-z)^{-\alpha}$ (used in the formula
(\ref{fact})) and the function  $(1-z)^{1-\alpha}$ (used in the
definition of $P_m$, formula (\ref{ppp})). Notice first, that both
$(1-z)^\beta$ is analytic in $\mathbb{D}$ for any $\beta\in (-1,1)$
and has the positive real part.  For $z=e^{i\theta}\in \mathbb{T}$,
we have
\[
(1-z)^\beta=\Bigl((1-\cos \theta)^2+\sin^2\theta\Bigr)^{\beta/2}\exp(-i\beta L(\theta))
\]
where
\[
L(\theta)=\arctan\left(\frac{\sin\theta}{1-\cos\theta} \right)
\]
and so
\begin{equation}\label{chisto}
(1-z)^\beta=|\theta|^{\beta}(1+O(\theta^2))\exp(-i\beta L(\theta)), \quad L(\theta)\to   \frac{\pi{\rm  sign}(\theta)(1-\theta+O(\theta^2))}{2}, \quad \theta\to 0
\end{equation}
We will now introduce the polynomials that will approximate $(1-z)^\beta$ uniformly on compacts in $\mathbb{D}$ and will
behave on the boundary in a controlled way. We will treat the cases of
positive and negative $\beta$ separately. Let $A_n(z)$ be the $n$-th Taylor
coefficients of $(1-z)^{\beta}$ plus a correction $Mn^{-\beta}$,
i.e.
\[
A_n(z)=Mn^{-\beta}+\sum_{j=1}^n c_j (1-z^j)
\]
and $M$ is some positive constant. The polynomial $R_{(n, -(1-\alpha))}$ in the main text will be taken as $A_n$ with $\beta=1-\alpha\in (0,1/2)$.

 For $B_n(z)$, we choose
$n$--th Taylor coefficient of $(1-z)^{-\beta}$, i.e.
\[
B_n(z)=1+\sum_{j=1}^{n} d_j z^j
\]
and
\[
d_j=\frac{\beta(\beta+1)\ldots(\beta+j-1)}{j!}=C_3(\beta)j^{\beta-1}+O(j^{\beta-2}), \quad C_3(\beta)>0
\]
The polynomial $R_{(n,\alpha/2)}$ used in the main text is $B_n$ with $\beta=\alpha/2\in (1/4,1/2)$.

We need the following simple lemmas.
\begin{lemma}\label{trifle}
For any $a>0$, we have
\[
\int_0^a \frac{\cos x}{x^\gamma}dx>0, \, {\rm if}\quad \gamma\in [1/2,1)
\]
and
\[
 \quad \int_0^a \frac{\sin
x}{x^\gamma}dx>0, \,\int_0^\infty \frac{\sin
x}{x^\gamma}dx>0,\,  {\rm if}\quad \gamma\in (0,1)
\]
\end{lemma}
\begin{proof}
The inequalities with $\sin$ are elementary as $x^{-\gamma}$ decays and $\sin x$ satisfies
\[\sin (\pi+x)=-\sin x; \quad \sin x>0, \,x\in (0,\pi)\]
 For the first inequality, we notice that
\[
\int_{3\pi/2}^a \frac{\cos x}{x^\gamma}dx>0
\]
for any $a>3\pi/2$ and we only need to show that
\[
\int_0^{3\pi/2} \frac{\cos x}{x^\gamma}dx>0
\]
Integrating by parts we have
\[
\int_0^{3\pi/2} \frac{\cos
x}{x^\gamma}dx=-\left(\frac{2}{3\pi}\right)^\gamma+\gamma\int_0^{3\pi/2}\frac{\sin
x}{x^{\gamma+1}}dx>-\left(\frac{2}{3\pi}\right)^\gamma+\frac{2\gamma}{\pi}
\int_0^{\pi/2}x^{-\gamma}dx
\]
where in the last inequality we dropped the integral over
$[\pi/2,3\pi/2]$ and used decay of $x^{-1}\sin x$ on $[0,\pi/2]$.
Calculating the integral, we get
\[
\frac{2\gamma}{\pi(1-\gamma)}\left(\frac{\pi}{2}\right)^{1-\gamma}-\left(\frac{2}{3\pi}\right)^\gamma>0
\]
for $\gamma\in [1/2,1)$.
\end{proof}

Let us first study the properties of $B_n$. As $B_n$ is the Taylor expansion of $(1-z)^{-\beta}$ and $\beta\in (0,1/2)$, we have the uniform convergence $B_n(z)\to (1-z)^{-\beta}$ in $\{|z|\leq 1\}\cap \{|1-z|>1-\upsilon\}$ for any fixed $\upsilon>0$ as long as $n\to \infty$. We now take $z=e^{i\theta}$ with $\theta\in (-\upsilon,\upsilon)$ where $\upsilon$ is small.\smallskip

We will need to use the following approximations by the integrals. Let $\gamma\in (0,1)$.
\begin{equation}\label{repa1}
\int_1^n \frac{\cos(x\theta)}{x^\gamma}dx=\sum_{j=1}^{n-1}
\int_{j}^{j+1} \frac{\cos(x\theta)}{x^\gamma}dx=\sum_{j=1}^{n-1}
\frac{1}{j^\gamma}\int_{j}^{j+1} {\cos(x\theta)}dx+\sum_{j=1}^{n-1}
\int_j^{j+1}
\cos(x\theta)\left(\frac{1}{x^\gamma}-\frac{1}{j^\gamma}\right)dx
\end{equation}
Since
\begin{equation}\label{motr}
\max_{x\in [j,j+1]} |x^{-\gamma}-j^{-\gamma}|\lesssim j^{-\gamma-1}
\end{equation}
the second term is $O(1)$ uniformly in $\theta$ and $n$ and that
gives
\[
\int_1^n \frac{\cos(x\theta)}{x^\gamma}dx=O(1)+\sum_{j=1}^{n-1}
\frac{1}{j^\gamma}
\frac{\sin(\theta/2)}{\theta/2}\cos(j\theta+\theta/2)=
\]
\[
O(1)+\sum_{j=1}^{n-1} \frac{1}{j^\gamma}
\frac{\sin(\theta/2)}{\theta/2}\Bigl(\cos(j\theta)\cos(\theta/2)-\sin(j\theta)\sin(\theta/2)\Bigr)
\]
Similarly
\begin{equation}\label{repa2}
\int_1^n \frac{\sin(x\theta)}{x^\gamma}dx=\sum_{j=1}^{n-1}
\int_{j}^{j+1} \frac{\sin(x\theta)}{x^\gamma}dx=\sum_{j=1}^{n-1}
\frac{1}{j^\gamma}\int_{j}^{j+1} {\sin(x\theta)}dx+\sum_{j=1}^{n-1}
\int_j^{j+1}
\sin(x\theta)\left(\frac{1}{x^\gamma}-\frac{1}{j^\gamma}\right)dx
\end{equation}
By (\ref{motr}), the second term is $\overline{o}(1)$ as $\theta\to
0$ uniformly in $n$. Therefore, we have
\[
\int_1^n
\frac{\sin(x\theta)}{x^\gamma}dx=\overline{o}(1)+\sum_{j=1}^{n-1}
\frac{1}{j^\gamma}\int_{j}^{j+1} {\sin(x\theta)}dx=
\]
\[
o(1)+\sum_{j=1}^{n-1} \frac{1}{j^\gamma}
\frac{\sin(\theta/2)}{\theta/2}\sin(j\theta+\theta/2)=
\]
\[
\overline{o}(1)+\sum_{j=1}^{n-1} \frac{1}{j^\gamma}
\frac{\sin(\theta/2)}{\theta/2}\Bigl(\sin(j\theta)\cos(\theta/2)+\cos(j\theta)\sin(\theta/2)\Bigr)
\]
Above $O(1)$ and $\overline{o}(1)$ are in $\theta\to 0$ uniformly in
$n$. Now, representations (\ref{repa1}) and (\ref{repa2}) yield the
formulas for
\[
\sum_{j=1}^{n-1}\frac{\cos(j\theta)}{j^\gamma}, \quad \sum_{j=1}^{n-1}\frac{\sin(j\theta)}{j^\gamma}
\]
i.e.
\begin{equation}\label{sin-a}
\sum_{j=1}^{n-1} \frac{\cos(j\theta)}{j^\gamma}=O(1)+C_{11}(\theta)\int_1^n \frac{\cos(x\theta)}{x^\gamma}dx+C_{12}(\theta)\int_1^n \frac{\sin(x\theta)}{x^\gamma}dx
\end{equation}
and
\begin{equation}\label{cos-a}
\sum_{j=1}^{n-1}
\frac{\sin(j\theta)}{j^\gamma}=\overline{o}(1)+C_{21}(\theta)\int_1^n
\frac{\sin(x\theta)}{x^\gamma}dx+C_{22}(\theta)\int_1^n
\frac{\cos(x\theta)}{x^\gamma}dx
\end{equation}
where $C_{11}\to 1, C_{12}\to 0, C_{21}\to 1, C_{22}\to 0$ as
$\theta\to 0$ uniformly in $n$.\smallskip

Now we are ready for the next lemma.
\begin{lemma}\label{poly2}
Let $\beta\in (0,1/2)$ and $\upsilon$ is sufficiently small fixed
positive number, then
\[
\Re B_n(e^{i\theta})\sim (n^{-1}+|\theta|)^{-\beta}, \quad \theta\in (-\upsilon,\upsilon)
\]
and
\[
 \quad \frac{\Im B_n(e^{i\theta})}{{\rm sign}(\theta)}\sim |\theta|^{-\beta}, \quad 0.01n^{-1}<|\theta|<\upsilon
\]
\[
 \quad \frac{\Im B_n(e^{i\theta})}{\theta}\sim n^{1+\beta}, \quad |\theta|<0.01n^{-1}
\]

\end{lemma}
\begin{proof}
 The case $|\theta|<0.01n^{-1}$
follows from $\cos(j\theta)\sim 1$ and $\sin(j\theta)/(j\theta)\sim 1$. For the other $\theta$, we first notice that it is sufficient to consider $\theta\in (0.01n^{-1},\upsilon)$ and that
\[
B_n(e^{i\theta})=1+C_3(\beta)\left(\sum_{j=1}^{n} j^{-1+\beta}e^{i\theta j}+O(1)\right)
\]
Let $\gamma=1-\beta\in (1/2,1)$ and use the formulas (\ref{sin-a}) and (\ref{cos-a}).
Notice that
\[
\int_1^n
\frac{\cos(x\theta)}{x^\gamma}dx=\theta^{\gamma-1}\int_\theta^{n\theta}\frac{\cos
t}{t^\gamma}dt\sim \theta^{\gamma-1} \quad ({\rm any}\, \gamma\in (1/2,1))
\]
as long as $\theta\in (0.01n^{-1},\upsilon)$ as follows from the lemma \ref{trifle}.
Similarly
\[
\int_1^n
\frac{\sin(x\theta)}{x^\gamma}dx=\theta^{\gamma-1}\int_\theta^{n\theta}\frac{\sin
t}{t^\gamma}dt\sim \theta^{\gamma-1}\quad ({\rm any}\, \gamma\in (0,1))
\]
That finishes the proof.
\end{proof}
\begin{lemma}\label{derider}
For any $\beta\in (0,1)$ we have
\[
|B_n'(e^{i\theta})|\lesssim \left\{
\begin{array}{cc}
|\theta|^{-1}n^\beta, & |\theta|>n^{-1}\\
n^{1+\beta}, & |\theta|<n^{-1}
\end{array}
\right.
\]
\[
|B''_n(e^{i\theta})|\lesssim \left\{
\begin{array}{cc}
|\theta|^{-1}n^{\beta+1}, & |\theta|>n^{-1}\\
n^{2+\beta}, & |\theta|<n^{-1}
\end{array}
\right.
\]
where the derivative is taken in $\theta\in (-\upsilon,\upsilon)$.
\end{lemma}

\begin{proof}
For $|\theta|<n^{-1}$, this follows from
\[
B_n'=\sum_{j=1}^n ijd_je^{ij\theta}, \,B_n''=\sum_{j=1}^n (ij)^2d_je^{ij\theta}
\]
by estimating the absolute values of the terms.

 For $|\theta|>1/n$
we can use Abel's lemma. Indeed,
\[
|B'_n|=\left|\sum_{j=1}^n jd_je^{ij\theta}\right|\lesssim \left|\sum_{j=1}^n e^{ij\theta}j^{\beta}\right|+\sum_{j=1}^n j^{\beta-1}
\]
The second term in the sum is bounded by $Cn^{\beta}$. For the first one, we have
\[
\left|\sum_{j=1}^n e^{ij\theta}j^{\beta}\right|\lesssim n^\beta|S_n|+\left|\sum_{j=1}^n S_j j^{\beta-1}\right|, \quad S_j=\sum_{j=1}^n e^{ij\theta}, \quad |S_j|\lesssim |\theta|^{-1}
\]
and that yields the bound for $B_n'$. The second derivative can be estimated similarly.
\end{proof}

Next, we will study the polynomial $A_n$. For the Taylor expansion of $(1-z)^\beta$, we have
\[
(1-z)^\beta=1+\sum_{j=1}^\infty \frac{(-1)^j\beta(\beta-1)\ldots(\beta-(j-1))}{j!}z^j
\]
If
\[
c_j=\frac{\beta(1-\beta)\ldots(j-1-\beta)}{j!}=C_4(\beta)j^{-\beta-1}+O(j^{-\beta-2})>0
\]
then
\[
\sum_{j=1}^\infty c_j=1
\]
Therefore, the formula for $A_n$ can be rewritten as
\[
A_n(z)=Mn^{-\beta}+\sum_{j=1}^n c_j(1-z^j)
\]
We again notice that $A_n(z)$ converges uniformly to $(1-z)^\beta$ in $\{|z|\leq 1\}\cap \{|1-z|>1-\upsilon\}$ for any fixed $\upsilon>0$.

\begin{lemma}\label{poly1}
Let $\beta\in (0,1)$. We have
\begin{equation}\label{odin}
\Re A_n(e^{i\theta})\sim (n^{-1}+|\theta|)^{\beta}, \quad \theta\in (-\upsilon,\upsilon)
\end{equation}
and
\begin{equation}\label{dva}
-\frac{\Im A_n(e^{i\theta})}{{\rm sign}( \theta)}\sim \left\{
\begin{array}{cc}
|\theta| n^{1-\beta}, & |\theta|<0.01n^{-1}\\ |\theta|^{\beta}\,, &0.01n^{-1}< |\theta|<\upsilon
\end{array}\right.
\end{equation}
\end{lemma}
\begin{proof}
We only need to handle positive $\theta$. Again, if $0<\theta<0.01n^{-1}$,
the estimate is simple.
\[
\Re A_n=Mn^{-\beta}+\sum_{j=1}^n c_j(1-\cos(j\theta))
\]
and we have a bound
\[
Mn^{-\beta}\leq \Re A_n\lesssim  Mn^{-\beta}+\sum_{j=1}^n j^{-\beta-1}(j^2\theta^2)\lesssim (M+1)n^{-\beta}
\]
Similarly
\[
\Im A_n=-\sum_{j=1}^n c_j\sin(j\theta)
\]
and
\[
\sum_{j=1}^n c_j\sin(j\theta)\sim \theta\sum_{j=1}^n jc_j\lesssim \theta n^{1-\beta}
\]
For the other $\theta$, we can again approximate
by the integrals. We have
\[
\sum_{j=1}^n c_j\sin(j\theta)=C_4(\beta)\sum_{j=1}^n j^{-\beta-1}\sin(j\theta)+O\left(\sum_{j=1}^n j^{-\beta-2}(j\theta)\right)
\]
The last term is $O(\theta)$. Then, take
\[
\int_1^n  \frac{\sin(x\theta)}{x^{1+\beta}}dx= \sum_{j=1}^{n-1} \int_{j}^{j+1}\frac{\sin(x\theta)}{x^{1+\beta}}dx=
\sum_{j=1}^{n-1}j^{-\beta-1}\int_{j}^{j+1}\sin(x\theta)dx+O(\theta)
\]
For the sum, we have
\[
 \sum_{j=1}^{n-1} j^{-\beta-1} \int_{j}^{j+1}{\sin(x\theta)}dx=\sum_{j=1}^{n-1} j^{-\beta-1} \frac{\sin(\theta/2)}{\theta/2}\sin(j\theta+\theta/2)=
\]
\[
\frac{\sin(\theta/2)\cos(\theta/2)}{\theta/2}\sum_{j=1}^{n-1} j^{-\beta-1}\sin(j\theta)+\frac{\sin^2(\theta/2)}{\theta/2}\sum_{j=1}^{n-1} j^{-\beta-1}\cos(j\theta)
\]
The second term is $O(\theta)$ and
\[
\frac{\sin(\theta/2)\cos(\theta/2)}{\theta/2}\sim 1
\]
for $\theta\in (0,\upsilon)$.
Then,
\[
\int_1^n \frac{\sin(x\theta)}{x^{1+\beta}}dx=\theta^\beta
\int_\theta^{n\theta} \frac{\sin x }{x^{1+\beta}}dx=\int_0^{n\theta} \frac{\sin x }{x^{1+\beta}}dx+O(\theta)
\]
Notice that
\[
C>\int_{0}^{a}\frac{\sin x }{x^{1+\beta}}dx>\delta_2>0
\]
for any $a>0.01$ and so we have
\[
\sum_{j=1}^nc_j\sin(j\theta)\sim \theta^{\beta}+O(\theta)\sim \theta^\beta
\]
This implies (\ref{dva}). For the real part,
\[
\Re A_n(e^{i\theta})=Mn^{-\beta}+T_n(\theta)+O\left(\theta\sum_{j=1}^n j^{-1-\beta}\right), \,T_n=\sum_{j=1}^n
\frac{1-\cos(j\theta)}{j^{1+\beta}}
\]
The last term is $O(\theta)$
and for $T_n$ we have
\[
T_n(0)=0, \quad T_n'(\theta)=\sum_{j=1}^n \frac{\sin(j\theta)}{j^\beta}
\]
For $\theta\in (0,0.01n^{-1})$, we have $T_n'\sim \theta n^{2-\beta}$. For $\theta\in (0.01n^{-1},\upsilon)$, the formula
(\ref{cos-a}) gives
\[
T_n'\sim \theta^{\beta-1}
\]
Integration gives
\[
T_n(\theta)=\int_0^\theta T_n'(\xi)d\xi \sim \theta^\beta, \quad \theta\in (0.01n^{-1},\upsilon)
\]
That finishes the proof.
\end{proof}
It is instructive to compare the results of lemmas \ref{poly2} and \ref{poly1} with (\ref{chisto}).\smallskip

For the derivative of $A_n$ in $\theta$, we have
\[
A_n'=-i\sum_{j=1}^n jc_je^{ij\theta}
\]

\begin{lemma}\label{der-der}
If $\beta\in (0,1)$, then
\[
|A_n'|\lesssim  \left\{
\begin{array}{cc}
|\theta|^{\beta-1}, & |\theta|>0.01n^{-1}\\
n^{1-\beta}, & |\theta|<0.01n^{-1}
\end{array}
\right.
\]
uniformly in $n$.
\end{lemma}
\begin{proof}
For $|\theta|<0.01n^{-1}$, the estimate is obtained by taking the absolute values in the sum. For $|\theta|>0.01n^{-1}$,
\[
|A_n'|\lesssim \left|\sum_{j=1}^n j^{-\beta} e^{ij\theta}\right|+1
\]
The estimates (\ref{sin-a}) and (\ref{cos-a}) along with the trivial estimates on the integrals involved yield the statement of the lemma.
\end{proof}
\section{Appendix B}

In this Appendix, we will control $\phi$, the phase of $Q_m(e^{i\theta})$, for $|\theta|<\upsilon$, where $\upsilon$ is some small positive and fixed number.
\begin{lemma}
For any $\theta\in (-\upsilon,\upsilon)$, we have
\[
|\phi'(\theta)|\lesssim m
\]
\end{lemma}
\begin{proof}   Recall that (see (\ref{mult-mult}))
\begin{equation}\label{mult-mult1}
Q_m(z)=\exp\left(\frac{1}{2\pi} \int_{-\pi}^\pi C(z,e^{i\xi})\log |Q_m(e^{i\xi})|d\xi\right)
\end{equation}
and $\phi(\theta)=\arg Q_m(e^{i\theta})$, i.e.
\[
\phi(\theta)=\Im \left(  \frac{1}{2\pi} \int_{-\pi}^\pi C(e^{i\theta},e^{i\xi})\log |Q_m(e^{i\xi})|d\xi\right)
\]
where
\[
C(e^{i\theta},e^{i\xi})=\frac{e^{i\xi}+e^{i\theta}}{e^{i\xi}-e^{i\theta}}
\]
where the integral is taken in principal value.
\[
\phi(\theta)=-\frac{1}{2\pi}\int_{-\pi}^\pi \frac{\cos((\xi-\theta)/2)}{\sin((\xi-\theta)/2)}\log|Q_m(e^{i\xi})|d\xi
\]
which amounts to controlling the Hilbert transform of $\log|Q_m(e^{i\xi})|$ since
\[
 \frac{\cos(\xi/2)}{\sin(\xi/2)}=\frac{2}{\xi}+O(\xi), \quad \left(\frac{\cos(\xi/2)}{\sin(\xi/2)}\right)'=-\frac{2}{\xi^2}+O(1)
\]
Therefore, if $x\in (-\upsilon,\upsilon)$, then
\[
|\phi'(x)|\lesssim  \left|\int_{-\pi}^\pi \frac{D_m'(\xi+x)}{\xi}d\xi\right|+1
\]
where
\[
D_m(\xi)=\log\left( \cal{G}_m(\xi) +|R_{(m,\alpha/2)}(e^{i\xi})|^2\right)+\log
m^{-1}
\]
(since $(\log m^{-1})'=0$).
So
\[
|\phi'(x)|\lesssim  1+m\left|\int_{-m\pi}^{m\pi} \frac{M_m'(t+\widehat{x})}{tM_m(t+\widehat x)}dt\right|, \quad M_m(t)=\exp D_m(t/m),
\quad \widehat{x}=mx
\]
and
\begin{eqnarray} \label{em-en}
M_m(t)=\frac{\sin^2 (t/2)}{m^2\sin^2(t/(2m))} +\frac{\cos^2
(t/2)}{2m^2\sin^2((t-\pi)/(2m))}+
\\
\frac{\cos^2(t/2)}{2m^2\sin^2((t+\pi)/(2m))}+
m^{-1}|R_{(m,\alpha/2)}(e^{it/m})|^2 \nonumber
\end{eqnarray}
due to (\ref{feya}) and (\ref{sdvig}).
Thus we only need to show that we
have
\[
I_1(\widehat x)=\left|\int_{-1}^1 \frac{M_m'(t+\widehat{x})}{tM_m(t+\widehat x)}dt\right|\lesssim 1
\]
and
\[
I_2(\widehat x)=\left|\int_{1<|t|<\pi m} \frac{M_m'(t+\widehat{x})}{tM_m(t+\widehat x)}dt\right|\lesssim 1
\]
uniformly in $\widehat x\in[-m\upsilon,m\upsilon]$.

Let us start with $I_2$ and so $|t|<\pi m$. Therefore, for $\xi=t+\widehat x$ we have $|\xi|<(\pi+\upsilon)m$. We can rewrite (\ref{em-en}) as follows
\begin{eqnarray}
M_m(\xi)=
\frac{\sin^2 (\xi/2)}{(\xi/2)^2}\left(1+G\left(\frac{\xi}{2m}\right)\right) +\frac{\cos^2
(\xi/2)}{2((\xi-\pi)/2)^2}\left(1+G\left(\frac{\xi-\pi}{2m}\right)\right)+\nonumber
\\
\frac{\cos^2
(\xi/2)}{2((\xi+\pi)/2)^2}\left(1+G\left(\frac{\xi+\pi}{2m}\right)\right)
+
m^{-1}|R_{(m,\alpha/2)}(e^{i\xi/m})|^2\label{uh-uh}
\end{eqnarray}
where
\[
G(x)=\frac{x^2}{\sin^2 x}-1
\]
is positive infinitely smooth function defined on $(-\pi,\pi)$ and
$G(x)\sim x^2$ on $(-a,a)\subset (-\pi,\pi)$.

For large $|\xi|$, we have
\begin{equation}\label{bolshoy}
M_m(\xi)
=\left(\frac{4}{\xi^2}+O(\xi^{-3})\right)
\left(1+O\left(\frac{\xi^2}{m^2}\right)\right)
+
m^{-1}|R_{(m,\alpha/2)}(e^{i\xi/m})|^2\gtrsim |\xi|^{-2}+m^{-1}|R_{(m,\alpha/2)}(e^{i\xi/m})|^2
\end{equation}
and
\begin{equation}\label{maliy}
M_m(\xi)>C+m^{-1}|R_{(m,\alpha/2)}(e^{i\xi/m})|^2
\end{equation}
uniformly in $m$ and  $\xi\in (-a,a)$ with any fixed $a$.
 For the derivative of $M_m(\xi)$, the representation (\ref{uh-uh}) gives an upper bound
\begin{equation}\label{last-term}
|M_m'(\xi)|\lesssim \frac{1}{(1+|\xi|)^3}+\frac{1}{m^2(|\xi|+1)}+\Bigl|(m^{-1}|R_{(m,\alpha/2)}(e^{i\xi/m})|^2)'\Bigr|
\end{equation}
\smallskip

Consider first $1<|\xi|<(\pi+\upsilon) m$. The lemma \ref{poly2} gives
\[
M_m(\xi)\gtrsim \frac{1}{\xi^2}+m^{-1}\left|\frac{\xi}{m}\right|^{-\alpha}
\]
Now, it is sufficient to use lemmas \ref{poly2} and  \ref{derider} to bound the last term in (\ref{last-term}) as
\[
\Bigl|(m^{-1}|R_{(m,\alpha/2)}(e^{i\xi/m})|^2)'\Bigr|\lesssim \frac{1}{m^2} \left|\frac{\xi}{m}\right|^{-\alpha/2}\frac{m^{1+\alpha/2}}{|\xi|}
\]
Combining the bounds we have
\begin{equation}\label{modin}
\left|\frac{M_m'(\xi)}{M_m(\xi)}\right|\lesssim
 \frac{
\displaystyle
 \frac{1}{|\xi|^3}+\frac{1}{m^2|\xi|}+\frac{m^{\alpha-1}}{|\xi|^{1+\alpha/2}}}
{
\displaystyle
\frac{1}{\xi^2}+\frac{m^{\alpha-1}}{|\xi|^\alpha}}
\leq
 \frac{
\displaystyle
 \frac{1}{|\xi|^3}+\frac{1}{m^2|\xi|}}
{
\displaystyle
\frac{1}{\xi^2}
}+
\frac{
\displaystyle
\frac{m^{\alpha-1}}{|\xi|^{1+\alpha/2}}}
{
\displaystyle
\frac{m^{\alpha-1}}{|\xi|^\alpha}}
\lesssim
|\xi|^{-1}+|\xi|^{\alpha/2-1}
\end{equation}
for $1<|\xi|<(\pi+\upsilon) m$.
\smallskip

For $|\xi|<1$, the analogous estimates give
\begin{equation}\label{modin1}
\left|\frac{M_m'(\xi)}{M_m(\xi)}\right|\lesssim
\frac{1+m^{\alpha-1}}{1+m^{\alpha-1}}\lesssim 1
\end{equation}
Combining (\ref{modin}) and (\ref{modin1}), we get
\begin{equation}\label{modin11}
\left|\frac{M_m'(\xi)}{M_m(\xi)}\right|\lesssim
(|\xi|+1)^{-1}+(|\xi|+1)^{\alpha/2-1}
\end{equation}
which holds on $|\xi|<(\pi+\upsilon)m$ uniformly.
Now, the Cauchy-Schwarz inequality implies the bound
for $I_2$
\[
|I_2(\widehat x)|\leq \left(  \int_{1<|t|<\pi m} \frac{dt}{t^2} \right)^{1/2}
\left(
\int_{|\xi|<(\pi+\upsilon)m} \left| \frac{M'(\xi)}{M(\xi)}\right|^2d\xi
\right)^{1/2}\lesssim 1
\]
\smallskip

Consider $I_1$. Apply the mean value formula to rewrite it as
\[
|I_1(\widehat x)|=\left|
\int_{-1}^1 \frac{1}{t}
\left(
\frac{M_n'(\widehat x)}{M_n(\widehat x)}+t\left(
\frac{M_n'(\xi)}{M_n(\xi)}
\right)'_{\xi=\xi_{\widehat x, t}}\right)
dt\right|\lesssim \left\| \frac{M_m''}{M_m}\right\|_\infty+\left\| \frac{M_m'}{M_m}\right\|^2_\infty
\]
The second term was estimated in (\ref{modin1}) so we only need to control the first one. We use (\ref{uh-uh}), (\ref{bolshoy}), and (\ref{maliy}) to get
\[
\left|\frac{M_m''(\xi)}{M_m(\xi)}\right|\lesssim \frac {
\displaystyle \frac{1}{(|\xi|+1)^2}+\left|  \left(
m^{-1}|R_{(m,\alpha/2)}(e^{i\xi/m})    |^2 \right)''\right| } {
\displaystyle
\frac{1}{1+\xi^2}+m^{-1}\left|R_{(m,\alpha/2)}(e^{i\xi/m})\right|^2
}
\]
The estimates from the lemmas \ref{derider} and \ref{poly2} in Appendix A can now be used as follows.
We have
\[
\left|\frac{M_m''(\xi)}{M_m(\xi)}\right|\lesssim 1+\left|
\frac{(R_{(m,\alpha/2)}(e^{i\xi/m}))'}{R_{(m,\alpha/2)}(e^{i\xi/m})}
\right|^2+\left|
\frac{(R_{(m,\alpha/2)}(e^{i\xi/m}))''}{R_{(m,\alpha/2)}(e^{i\xi/m})}
\right|
\]
For $1<|\xi|<\upsilon m+1$, one gets
\[
\left|
\frac{(R_{(m,\alpha/2)}(e^{i\xi/m}))'}{R_{(m,\alpha/2)}(e^{i\xi/m})}
\right|\lesssim |\xi|^{\alpha/2-1}, \quad
\left|
\frac{(R_{(m,\alpha/2)}(e^{i\xi/m}))''}{R_{(m,\alpha/2)}(e^{i\xi/m})}
\right|\lesssim |\xi|^{\alpha/2-1}
\]
For $|\xi|<1$, we have
\[
\left|
\frac{(R_{(m,\alpha/2)}(e^{i\xi/m}))'}{R_{(m,\alpha/2)}(e^{i\xi/m})}
\right|\lesssim 1,\quad
\left|
\frac{(R_{(m,\alpha/2)}(e^{i\xi/m}))''}{R_{(m,\alpha/2)}(e^{i\xi/m})}
\right|\lesssim 1
\]
This gives a bound
\[
 \left\| \frac{M_m''}{M_m}\right\|_\infty\lesssim 1
\]
which ensures
\[
|I_1(\widehat x)|\lesssim 1
\]
uniformly in $\widehat x\in[-m\upsilon,m\upsilon]$. This finishes the proof of the lemma.
\end{proof}

\bigskip

{\bf Remark.} Since $Q_m$ has no zeroes in $\overline{\mathbb{D}}$, we can write it as
\[
Q_m(z)=\prod_{j=1}^m (z-z_j)
\]
where $z_j\in \{|z|>1\}$ are its zeroes. Therefore,
\[
\arg Q_m(e^{i\theta})=\sum_{j=1}^m \Im \log(e^{i\theta}-z_j)
\]
Taking the derivative, we have
\[
\phi'=\sum_{j=1}^m \Im \left(  \frac{ie^{i\theta}}{e^{i\theta}-z_j}  \right)
\]
This formula shows that the size of $\phi'$ is controlled by the location of zeroes with respect to the unit circle.

\end{document}